\documentclass[11pt,reqno]{amsart}

\usepackage[margin=1.05in]{geometry}
\usepackage{amsmath,amssymb,mathtools,mathrsfs}
\usepackage{microtype}
\usepackage{enumitem}
\usepackage{aliascnt}
\usepackage{xcolor}
\usepackage[colorlinks=true,linkcolor=blue!55!black,citecolor=blue!55!black,urlcolor=blue!55!black]{hyperref}
\hypersetup{
  pdftitle={Sharp nonuniqueness for singular solutions of the one-dimensional periodic cubic NLS},
  pdfauthor={Qirui Peng},
  pdfsubject={Sharp nonuniqueness and singular solutions for periodic cubic NLS},
  pdfkeywords={cubic NLS, non-uniqueness, convex integration, intermittent slab}
}
\usepackage[nameinlink,capitalize,noabbrev]{cleveref}

\allowdisplaybreaks
\numberwithin{equation}{section}
\setlist[enumerate]{leftmargin=2.2em,itemsep=0.3em,topsep=0.4em}
\setlist[itemize]{leftmargin=2.0em,itemsep=0.3em,topsep=0.4em}

\newtheorem{theorem}{Theorem}[section]

\newaliascnt{proposition}{theorem}
\newtheorem{proposition}[proposition]{Proposition}
\aliascntresetthe{proposition}

\newaliascnt{lemma}{theorem}
\newtheorem{lemma}[lemma]{Lemma}
\aliascntresetthe{lemma}

\newaliascnt{corollary}{theorem}

\aliascntresetthe{corollary}

\newaliascnt{claim}{theorem}

\aliascntresetthe{claim}

\theoremstyle{definition}
\newaliascnt{definition}{theorem}
\newtheorem{definition}[definition]{Definition}
\aliascntresetthe{definition}

\newaliascnt{construction}{theorem}

\aliascntresetthe{construction}

\theoremstyle{remark}
\newaliascnt{remark}{theorem}
\newtheorem{remark}[remark]{Remark}
\aliascntresetthe{remark}

\crefname{theorem}{Theorem}{Theorems}
\Crefname{theorem}{Theorem}{Theorems}
\crefname{proposition}{Proposition}{Propositions}
\Crefname{proposition}{Proposition}{Propositions}
\crefname{lemma}{Lemma}{Lemmas}
\Crefname{lemma}{Lemma}{Lemmas}
\crefname{corollary}{Corollary}{Corollaries}
\Crefname{corollary}{Corollary}{Corollaries}
\crefname{claim}{Claim}{Claims}
\Crefname{claim}{Claim}{Claims}
\crefname{definition}{Definition}{Definitions}
\Crefname{definition}{Definition}{Definitions}
\crefname{construction}{Construction}{Constructions}
\Crefname{construction}{Construction}{Constructions}
\crefname{remark}{Remark}{Remarks}
\Crefname{remark}{Remark}{Remarks}

\newcommand{\T}{\mathbb T}
\newcommand{\R}{\mathbb R}
\newcommand{\Z}{\mathbb Z}
\newcommand{\C}{\mathbb C}
\newcommand{\A}{\mathcal A}
\newcommand{\CS}{\mathfrak C}
\newcommand{\N}{\mathcal N}
\newcommand{\supp}{\operatorname{supp}}

\newcommand{\wt}[1]{\langle #1\rangle}
\newcommand{\dd}{\,\mathrm d}
\newcommand{\one}{\mathbf 1}
\newcommand{\eps}{\varepsilon}

\newcommand{\Pnz}{\mathbb P_{\neq 0}}

\def \l {\lambda}

\title[Sharp nonuniqueness for cubic NLS]{Sharp nonuniqueness for singular solutions of the one-dimensional periodic cubic NLS}
\author{Qirui Peng}
\address{Department of Mathematics, University of California, Santa Barbara, Santa Barbara, CA 93106, USA}
\email{qpeng9@ucsb.edu}
\date{August 2026}

\begin{document}

\begin{abstract}
We construct by convex integration a nonzero singular weak solution of the one-dimensional periodic cubic nonlinear Schr\"odinger equation that is compactly supported in time, has zero initial datum, and satisfies
\[
u\in\bigcap_{\alpha<1/6} C_t^0 H_x^\alpha
\cap
\bigcap_{1\leq p<3} C_t^0 L_x^p.
\]
The key new ingredient is a two-carrier perturbation: its zero-output interaction cancels the previous error, without introducing any zero Fourier mode, which is crucial in the cubic NLS. Whenever a singular solution belongs to \(C_t^0L_x^3\), this nonlinearity agrees with the ordinary product \(|u|^2u\). Together with unconditional uniqueness at \(H^{1/6}\), this makes the threshold sharp within the singular solution class considered here.
\end{abstract}

\maketitle

\section{Introduction}\label{sec:intro}

We consider the Cauchy problem for the periodic cubic nonlinear Schr\"odinger equation
\begin{equation}\label{eq:NLS}
 \begin{cases}
  i\partial_tu+\partial_x^2u+\sigma|u|^2u=0
  &\text{on }(0,1)\times\T,\\
  u(0,\cdot)=u_{\mathrm{initial}}
  &\text{on }\T,
 \end{cases}
 \qquad \sigma\in\{-1,1\},
\end{equation}
where $u:[0,1]\times\T\to\C$, $\T=\R/\Z$, and $u_{\mathrm{initial}}$ is the prescribed initial datum.

Equation~\eqref{eq:NLS} is a canonical Hamiltonian envelope model for weakly nonlinear dispersive waves. It arises as a modulation equation for nearly monochromatic gravity waves on deep water \cite{ZakharovWaterWaves} and for optical pulses in dielectric fibers \cite{HasegawaTappertI,HasegawaTappertII}. With the sign convention in \eqref{eq:NLS}, $\sigma=1$ is focusing and $\sigma=-1$ is defocusing. The inverse-scattering theories for the focusing and defocusing equations were established in the classical works of Zakharov and Shabat \cite{ZakharovShabatFocusing,ZakharovShabatDefocusing}. For smooth solutions, both signs conserve the mass and Hamiltonian
\begin{equation}\label{eq:intro-mass-energy}
 M(u)=\int_{\T}|u|^2\,\dd x,
 \qquad
 H_\sigma(u)=\int_{\T}\left(|\partial_xu|^2-\frac{\sigma}{2}|u|^4\right)\dd x.
\end{equation}
Our construction treats both signs simultaneously: $\sigma$ affects only the algebraic choice of one carrier amplitude, while the analytic estimates are identical.

The one-dimensional cubic NLS is mass-subcritical, with scaling index $-1/2$. Its classical $L^2$ theory was established by Tsutsumi on $\R$ and Bourgain on $\T$ \cite{Tsutsumi,Bourgain}. In the periodic setting, Bourgain's Fourier restriction method gives a global solution in $C_t^0L_x^2$, with uniqueness in the auxiliary restriction space used in the construction. Guo, Kwon, and Oh subsequently proved unconditional global well-posedness in $H^\alpha(\T)$ for every $\alpha\geq1/6$ \cite{GKO}. The same exponent is the Sobolev threshold at which the ordinary cubic is automatically a spacetime distribution for every trajectory in $C_t^0H_x^\alpha$:
\[
 H^{1/6}(\T)\hookrightarrow L^3(\T)
 \qquad\Longrightarrow\qquad
 |u|^2u\in C_t^0L_x^1.
\]
The present paper asks how far nonuniqueness can penetrate below this endpoint when the cubic is defined intrinsically through its Fourier interactions rather than assumed a priori to be the ordinary pointwise product.

\begin{theorem}\label{thm:main}
Fix $\sigma\in\{-1,1\}$. There exists a nonzero function $u$ with the following properties:
\begin{enumerate}[label=\textup{(\roman*)}]
 \item $u$ is compactly supported in time in $(0,1)$ and $u(0,\cdot)=0$;
 \item
 \begin{equation}\label{eq:main-regularity}
 u\in\bigcap_{\alpha<1/6}C_t^0H_x^\alpha
 \cap
 \bigcap_{1\leq p<3}C_t^0L_x^p;
 \end{equation}
 \item $u$ solves \eqref{eq:NLS} with $u_{\mathrm{initial}}=0$ in the sense of \cref{def:singular-solution}.
\end{enumerate} 
\end{theorem}

\begin{remark}\label{rem:no-mass-conservation}
The solution in \cref{thm:main} belongs to $C_t^0L_x^2$, but at that regularity the singular cubic $\N(u)$ defined by \cref{def:cubic-product} need not agree with the ordinary product $|u|^2u$. In fact, the equation only places $\N(u)$ in $C_t^0H_x^{-s}$ with $s>3$, so the pairing $\langle\N(u),u\rangle$ used in the usual derivation of mass conservation is not defined under the stated assumptions. Thus the standard $L^2$ conservation argument does not apply to this singular solution. 
\end{remark}

\begin{remark}\label{rem:sharpness-of-1/6}
The singular solution $u$ becomes a classical weak solution once $u\in C_t^0L_x^3$, by \cref{lem:compatibility}. Therefore the singular weak solution in the sense of \cref{def:singular-solution} that belongs to $C_t^0H_x^{1/6}$ is unique. Consequently, the threshold $1/6$ is sharp within this singular solution class.
\end{remark}

\subsection{Rough solution theories and generalized nonlinearities}

Below $L^2$, the full periodic equation exhibits markedly different behavior. Molinet proved that for every $s<0$ the periodic flow map fails to be continuous from $H^s(\T)$ to itself at any fixed nonzero time \cite{Molinet}. Guo and Oh proved a stronger non-existence statement for data in $H^s(\T)\setminus L^2(\T)$ when $-1/8<s<0$, in the sense of limits of smooth solutions to the full equation \cite{GuoOh}. Colliander and Oh proved almost sure local well-posedness for the modified cubic NLS for randomized data in $H^\alpha(\T)$ when $\alpha>-1/3$, and almost sure global well-posedness when $\alpha>-1/12$ \cite{CollianderOh}. Exploiting the integrable structure, Oh and Wang proved global well-posedness of the modified periodic equation in $\mathcal F L^p(\T)$ for every finite $p$ \cite{OhWangGWP}; their later infinite normal-form analysis gives an unconditionally globally well-posed normal-form equation for all $1\leq p<\infty$ and unconditional uniqueness for the modified cubic NLS when $1\leq p\leq3/2$ \cite{OhWangNF}.

A different route is to enlarge the notion of solution. Christ introduced a robust Fourier-cutoff definition of nonlinear products and, for the modified periodic cubic NLS, constructed nontrivial generalized solutions with zero initial datum in $C_t^0H_x^\alpha$ for every $\alpha<0$ \cite{Christ}. Guo, Kwon, and Oh subsequently formulated the corresponding Fourier-cutoff solution class for the full periodic cubic equation and constructed canonical solutions in $C_t^0L_x^2$ \cite{GKO}. The singular solution class used here is narrower: the signed cubic interactions are absolutely summable in a weighted Fourier norm. Consequently, the nonlinearity is intrinsic and every admissible Fourier cutoff converges to the same limit. The construction reaches every $\alpha<1/6$, including positive Sobolev regularities and $C_t^0L_x^2$.

The compatibility result in \cref{lem:compatibility} identifies $\N(u)$ with $|u|^2u$ whenever $u\in C_t^0L_x^3$. Combined with $H^{1/6}(\T)\hookrightarrow L^3(\T)$ and the unconditional uniqueness theorem of Guo, Kwon, and Oh \cite{GKO}, this gives the endpoint assertion in \cref{thm:main}. Thus the word ``sharp'' refers specifically to the singular solution class introduced in \cref{def:singular-solution}.

\subsection{Convex integration and the new NLS obstruction}

Convex integration originated in the work of Nash and Kuiper on $C^1$ isometric embeddings and was developed into a general flexibility theory by Gromov \cite{Nash,Kuiper,Gromov}. De Lellis and Sz\'ekelyhidi brought this viewpoint into fluid dynamics \cite{DeLellisSzekelyhidi}; later oscillatory and intermittent schemes led to Isett's proof of the flexible half of Onsager's conjecture \cite{Isett}, the finite-energy Navier--Stokes nonuniqueness theorem of Buckmaster and Vicol \cite{BuckmasterVicol}, and the sharp result of Cheskidov and Luo in $L_t^pL_x^\infty$ for $p<2$ \cite{CheskidovLuo}.

A recent branch of the subject uses convex integration in spaces where the nonlinear product is itself singular, so the solution concept becomes part of the construction. Ashkarian, Bhargava, Gismondi, and Novack constructed intermittent stationary solutions of the two-dimensional Navier--Stokes equations at the sharp $L^2$ threshold \cite{ABGN}. Cheskidov and Hou constructed stationary singular Navier--Stokes solutions in every Besov space of negative smoothness and used them to prove failure of unconditional uniqueness in those spaces \cite{CheskidovHou}. In the dispersive setting, Pathak studied the existence of nontrivial stationary solutions of KdV up to the sharp $L^p$ space \cite{pathak2026nontrivial}.  Gismondi, Ma, Pathak, and Radu constructed nonunique solutions for periodic generalized KdV equations, together with an absolute-Fourier definition of the rough nonlinearity \cite{GMPR}.

The present argument uses the analytic slab from \cite{GMPR} and \cite{pathak2026nontrivial}, but it is not a direct transplantation of the KdV construction. The full cubic NLS presents three linked obstructions. First, the second-order dispersive error is too large in the $L^1$ space of a traditional convex integration scheme. Second, a one-amplitude cancellation results in a nonsmooth complex cubic root at the zeros of the error. Third, shifting that cubic relation away from zero creates a spatial constant. The outer derivative removes this term in KdV, whereas the full NLS retains it. The two-carrier polarization introduced below resolves the last two obstructions, while a negative Wiener space resolves the first.

\subsection{An obstruction to the standard convex integration scheme}

There is a simple scaling obstruction to applying a standard fluid-style convex integration scheme directly to a dispersive equation. We record the cubic version of the heuristic from \cite[Section~1.3]{GMPR}. Consider the relaxed equation 
\[
 i\partial_tu_q+\partial_x^2u_q+\sigma|u_q|^2u_q=E_q.
\]
 Let $\l_q=a^{b^q}$ with $a,b>1$, and suppose the stage-$q$ error has size
\[
 \|E_q \|_{L^1} \lesssim \delta_{q+1} :=\l_{q}^{-3\gamma},
 \qquad \gamma>0.
\]
A conventional cubic perturbation $w_{q+1}$ that cancels an error of size $\delta_{q+1}$ must have endpoint size $\|w_{q+1}\|_{L^3}^3\sim\delta_{q+1}$. If it is concentrated at the fully intermittent scale and frequency localized near $\l_{q+1}$, the natural estimates are
\begin{equation}\label{eq:intro-traditional-scaling}
 \|w_{q+1}\|_{L^1}
 \lesssim
 \delta_{q+1}^{1/3}\l_{q+1}^{-2/3}
 =\l_{q+1}^{-\gamma-2/3},
 \qquad
 \|w_{q+1}\|_{L^2}
 \lesssim
 \l_{q+1}^{-\gamma-1/6}.
\end{equation}
The second estimate would formally give
\[
 \|w_{q+1}\|_{H^\alpha}
 \lesssim
 \l_{q+1}^{\alpha-\gamma-1/6},
\]
so a traditional scheme would appear capable of reaching $\alpha<1/6+\gamma$. However, the dispersive term destroys this apparent gain. If the new error is measured in $L^1$, then
\[
 \|\partial_x^2w_{q+1}\|_{L^1}
 \lesssim
 \l_{q+1}^2\|w_{q+1}\|_{L^1}
 \lesssim
 \l_{q+1}^{4/3-\gamma}. 
\]
To make this no larger than the next error $\delta_{q+2}=\l_{q+2}^{-3\gamma}=\l_{q+1}^{-3b\gamma}$, one would need
\[
 \frac43-\gamma\leq-3b\gamma,
 \qquad\text{equivalently}\qquad
 \gamma\leq\frac{4}{3(1-3b)}<0,
\]
contradicting $\gamma>0$. 

The way around the obstruction is to change the error topology. We measure the residual in an inhomogeneous negative Wiener space $\A^{-s}$ with $s>3$, which is natural in the class of singular solutions (see \cref{{def:singular-solution}}). In this norm, the two derivatives are absorbed by the negative Fourier weight:
\[
 \|\partial_x^2 E_q \|_{\A^{-s}_{x,t}}
 \lesssim_s
 \|E_q \|_{C_t^0L_x^1}.
\]
\subsection{The main idea of the convex integration scheme}
The idea of convex integration is to introduce highly oscillatory terms in order to remove the present error through nonlinear interactions. However, the NLS algebra does not permit a direct one-amplitude cancellation. Naively, if $\rho$ is real, highly oscillatory, and such that $\rho^3 = 1$, then introducing a perturbation $w=a\rho$, where $a$ is the amplitude, produces $|w|^2w=|a|^2a$. Solving $|a|^2a=-\sigma E$ requires the inverse map
\[
 E\longmapsto |E|^{1/3}e^{i\arg(-\sigma E)},
\]
which is not differentiable at the zeros of the complex error $E$. Shifting the cubic relation away from zero by adding a constant $A>0$, as in
\[
(A+E)^{\frac{1}{3}}, \ \ \ A > \|E\|_{L^\infty},
\]
removes that singularity but leaves a spatial constant. In many equations, such as the KdV equations, an outer spatial derivative acting on the full nonlinearity annihilates that constant, while in \eqref{eq:NLS} it survives.

One way to overcome this obstruction is to use a two-carrier complex polarization. For $a>0$, $b\in\C$, and $z=e^{2\pi iKx}$, one has
\begin{equation}\label{eq:intro-polarization}
 |az+bz^2|^2(az+bz^2)
 =a^2\overline b+a(a^2+2|b|^2)z+b(2a^2+|b|^2)z^2+ab^2z^3.
\end{equation}
Given a complex error $E_q$, we take
\begin{equation}\label{eq:intro-b}
 b_{q+1}=-\sigma a_{q+1}^{-2}\overline{E_q}.
\end{equation}
Then $a_{q+1}^2\overline{b_{q+1}}=-\sigma E_q$, so the zero slab mode of the low-output interaction reproduces the complete error, and the nonzero slab modes are small in the Wiener space.

Given $(u_q, E_q)$ solving the relaxed equation, we would like to construct another solution $(u_{q+1},E_{q+1})$ to the same equation:
\[
 i\partial_tu_{q+1}+\partial_x^2u_{q+1}+\sigma|u_{q+1}|^2u_{q+1}=E_{q+1}
\]
with $E_{q+1}$ much smaller than $E_q$ in the negative Wiener norm. To achieve this, we add a high-frequency perturbation to $u_q$, i.e.,
\[
u_{q+1} = u_q + \omega_{q+1}.
\]
The carrier waves alone do not produce an increment that is sufficiently
small in $L^p$ or $H^\alpha$. To make the perturbation small simultaneously in these two spaces, we modulate it by an intermittent profile. 
The analytic building block is adapted from the $L^k$-normalized pulse trains of Gismondi, Ma, Pathak, and Radu \cite{GMPR}. In the cubic case, the profile has
\[
 \|\Theta\|_{L^1}\ll1,
 \qquad
 \|\Theta\|_{L^2}^2\ll1,
 \qquad
 \int_{\T}\Theta^3\,\dd x=1,
\]
and can be chosen with nonnegative Fourier coefficients on a sparse lattice. These features provide both intermittency and error control of low-mode outputs.

The perturbation is therefore given by
\begin{equation}\label{eq:intro-w}
 w_{q+1}
 =h_{q+1}(t)\Theta_{q+1}(x)
 \left(a_{q+1}e^{2\pi iK_{q+1}x}
 +b_{q+1}(t,x)e^{4\pi iK_{q+1}x}\right).
\end{equation}
Here $h_{q+1}$ is a temporal cutoff and $\Theta_{q+1}$ is the intermittent slab mentioned above. If the frequency parameter of $\Theta_{q+1}$ is $\lambda_{q+1}$ and its intermittency exponent is $\eps_{q+1}$, then
\[
 \|w_{q+1}\|_{L^p}
 \lesssim_q
 \lambda_{q+1}^{(1-\eps_{q+1})(1/3-1/p)}.
\]
The implicit constant is fixed after the old stage is frozen, so the increment is small for $p<3$. Its Fourier support lies at frequencies comparable to
\[
 K_{q+1}\simeq \nu_{q+1}=\lambda_{q+1}^{1+\eps_{q+1}},
\]
whereas
\[
 \|\Theta_{q+1}\|_{L^2}
 \lesssim
 \lambda_{q+1}^{-\frac16(1-\eps_{q+1})}.
\]
Consequently,
\[
 \|w_{q+1}\|_{H^\alpha}
 \lesssim_q
 \lambda_{q+1}^{(1+\eps_{q+1})\alpha-\frac16(1-\eps_{q+1})},
\]
which decays whenever
\[
 \alpha<\frac{1-\eps_{q+1}}{6(1+\eps_{q+1})}.
\]
Sending $\eps_{q+1}$ to zero along the iteration reaches every fixed $\alpha<1/6$.

Apart from the error estimate, one must also control the absolute
sum of the Fourier interactions of the cubic product, so that the limiting solution has a well-defined nonlinearity. Recall that
\[
 u_{q+1}=u_q+w_{q+1}.
\]
At each finite stage, the cubic nonlinearity expands as
\[
\begin{aligned}
 |u_{q+1}|^2u_{q+1}
 ={}&|u_q|^2u_q\\
 &+\underbrace{
 w_{q+1}\overline{u_q}u_q
 +u_q\overline{w_{q+1}}u_q
 +u_q\overline{u_q}w_{q+1}
 }_{\text{exactly one new factor}}\\
 &+\underbrace{
 w_{q+1}\overline{w_{q+1}}u_q
 +w_{q+1}\overline{u_q}w_{q+1}
 +u_q\overline{w_{q+1}}w_{q+1}
 }_{\text{exactly two new factors}}\\
 &+\underbrace{
 w_{q+1}\overline{w_{q+1}}w_{q+1}
 }_{\text{three new factors}}.
\end{aligned}
\]
Write the perturbation in \eqref{eq:intro-w} as
\[
 w_{q+1}=w_{q+1}^{(1)}+w_{q+1}^{(2)},
\]
where the packet \(w_{q+1}^{(j)}\) is centered at the carrier
frequency \(jK_{q+1}\), \(j\in\{1,2\}\). More precisely, each of its
frequencies has the form
\[
 jK_{q+1}+\eta+\xi, \qquad \eta\in\mu_{q+1}\mathbb Z,
 \qquad |\xi| \ll \l_{q+1},
\]
where \(\eta\) is a slab frequency of $\Theta_{q+1}$ and \(\xi\) is an old-frequency
modulation of $E_q$. Since the cubic output is \(n_1-n_2+n_3\), every term
with exactly one new factor retains a nonzero carrier
\(\pm jK_{q+1}\) and is therefore high.

With two new factors, the three possible placements have
carrier coefficients
\[
 \begin{array}{c|ccc}
 \text{new slots}
 &(1,2)&(1,3)&(2,3)\\
 \hline
 \text{carrier coefficient}
 &j_1-j_2&j_1+j_3&-j_2+j_3.
 \end{array}
\]
Here, for example, $(1,2)$ stands for the term $\omega_{q+1} \overline{\omega_{q+1}} u_q$. Thus a low carrier occurs only in the placements \((1,2)\) and
\((2,3)\), when the conjugated and nonconjugated factors come from
the same packet. After the carrier cancels, the remaining slab
frequency is \(r=\eta_1-\eta_2\) or
\(r=-\eta_2+\eta_3\). Its zero-frequency contribution costs
\[
 \widehat{\Theta_{q+1}^2}(0)
 =\int_{\T}\Theta_{q+1}^2\,\dd x,
\]
whereas \(r\neq0\) implies \(|r|\geq\mu_{q+1}\) and hence gains the
negative weight \(\langle r\rangle^{-s}\lesssim\mu_{q+1}^{-s}\).

Finally, with three new factors the carrier has Fourier mode
\((j_1-j_2+j_3)K_{q+1}\), and
\[
 j_1-j_2+j_3=0
 \quad\Longleftrightarrow\quad
 (j_1,j_2,j_3)=(1,2,1).
\]
For this special low output, we have 
\[
 w_{q+1}^{(1)}
 \overline{w_{q+1}^{(2)}}
 w_{q+1}^{(1)}
 =-\sigma E_q\Theta_{q+1}^3.
\]
Since
\[
 \widehat{\Theta_{q+1}^3}(0)
 =\int_{\T}\Theta_{q+1}^3\,\dd x=1,
\]
its zero slab frequency contributes exactly
\(\|E_q\|_{\A^{-s}}\), while its nonzero slab frequencies again gain
\(\mu_{q+1}^{-s}\). The non-negativity of the Fourier coefficients of
\(\Theta_{q+1}\) allows these convolution identities to remain valid
after absolute values are inserted term by term.

The rest of the paper is organized as follows. In \cref{sec:prelim} we fix the Fourier and duality conventions, define the weighted Wiener spaces and the singular solution class, and prove compatibility with the ordinary cubic, as well as the Fourier cutoff robustness. In \cref{sec:main-prop} we state the main inductive proposition and derive \cref{thm:main}. The intermittent profile and its Fourier, moment, and carrier-wave properties are developed in \cref{sec:slabs}. Finally, \cref{sec:proof-main-prop} proves the inductive proposition through the two-carrier construction, error estimates, signed absolute-Fourier analysis, exact block argument, and parameter choice.

\section{Preliminaries}\label{sec:prelim}

\subsection{Fourier Series}

For $f\in\mathcal D'(\T)$, we use the Fourier convention
\begin{equation}\label{eq:Fourier-convention}
 \widehat f(n):=\left\langle f,e^{2\pi i n x}\right\rangle,
 \qquad
 f(x)=\sum_{n\in\Z}\widehat f(n)e^{2\pi i n x}
\end{equation}
in the sense of distributions.  The bracket $\langle\cdot,\cdot\rangle$
stands for the standard duality pairing. For functions,
\[
 \langle f,g\rangle
 =\int_{\T}f(x)\overline{g(x)}\,\dd x.
\]
Thus, when $f\in L^1(\T)$,
\begin{equation}\label{eq:Fourier-integral}
 \widehat f(n)=\int_{\T}f(x)e^{-2\pi i n x}\,\dd x.
\end{equation}
For $s>0$, $f\in H_x^{-s}$, and $g\in H_x^s$, we write
\begin{equation}\label{eq:Sobolev-duality}
 \left\langle f,g\right\rangle_{H_x^{-s},H_x^s}
 :=\sum_{n\in\Z}\widehat f(n)\overline{\widehat g(n)}.
\end{equation}
For smooth functions it agrees with
$\int_{\T}f(x)\overline{g(x)}\,\dd x$.

Following Christ and Guo--Kwon--Oh \cite{Christ,GKO}, a \emph{sequence of Fourier cutoff operators} is a sequence $(P_N)_{N\geq1}$ of Fourier multipliers
\begin{equation}\label{eq:Fourier-cutoff}
 \widehat{P_Nf}(n)=m_N(n)\widehat f(n)
\end{equation}
such that each $m_N$ has finite support,
\begin{equation}\label{eq:cutoff-uniform}
 \sup_N\|m_N\|_{\ell^\infty(\Z)}<\infty,
 \qquad
 m_N(n)\longrightarrow1
 \quad\text{for every fixed }n\in\Z.
\end{equation}
We also write $\mathbb P_{\neq0}f=f-\widehat f(0)$ and use $\mathbb P_{\leq N}$ for a standard smooth Fourier projection with symbol $\vartheta(n/N)$, where $\vartheta\in C_c^\infty(\R)$ is even and equals one near the origin.

\subsection{Weighted Wiener spaces}

We introduce the following weighted inhomogeneous Wiener spaces below.  This formulation, including the stronger spacetime norm obtained by summing the $C_t^0$ norms of the Fourier coefficients, follows the setup used for intermittent slabs in \cite{GMPR}; see also \cite{Katznelson} for the classical Wiener algebra.

\begin{definition}[Weighted Wiener spaces]\label{def:Wiener-spaces}
Let $r\in\R$.  Define
\begin{equation}\label{eq:Ar-def}
 \A^r(\T)
 :=\left\{f\in\mathcal D'(\T):
 \|f\|_{\A_x^r}:=
 \sum_{n\in\Z}\wt n^r|\widehat f(n)|<\infty\right\},
 \qquad
 \wt n=(1+n^2)^{1/2}.
\end{equation}
For $f\in C_t^0\mathcal D'_x([0,1]\times\T)$, define
\begin{equation}\label{eq:Ar-xt-def}
 \|f\|_{\A_{x,t}^r}
 :=\sum_{n\in\Z}\wt n^r
 \|\widehat f(\cdot,n)\|_{C_t^0}.
\end{equation}
The corresponding Banach space is denoted by $\A_{x,t}^r([0,1]\times\T)$.
\end{definition}

The norm in \eqref{eq:Ar-xt-def} is stronger than the Bochner norm:
\begin{equation}\label{eq:Wiener-embedding}
 \|f\|_{C_t^0\A_x^r}
 =\sup_{t\in[0,1]}\sum_n\wt n^r|\widehat f(t,n)|
 \leq\|f\|_{\A_{x,t}^r}.
\end{equation}
Since $\ell^1(\Z)\hookrightarrow\ell^2(\Z)$, one also has
\begin{equation}\label{eq:Wiener-Sobolev-embedding}
 \|f\|_{C_t^0H_x^r}
 \leq \|f\|_{C_t^0\A_x^r}
 \leq \|f\|_{\A_{x,t}^r}
 \qquad\text{for every }r\in\R.
\end{equation}
The inhomogeneous weight is essential here, as the zero Fourier mode of the cubic NLS is part of the equation.

We record the elementary Wiener estimates used throughout this paper. The first estimate is the inhomogeneous counterpart of the Kato--Ponce-type weighted Wiener product estimate in \cite[Lemma~2.10]{GMPR}.

\begin{lemma}\label{lem:Wiener}
Let $s\geq0$. We have the following estimates:
\begin{enumerate}[label=\textup{(\roman*)}]
 \item For time-dependent trigonometric polynomials $f,g$,
 \begin{equation}\label{eq:Wiener-product}
  \|fg\|_{\A^{-s}_{x,t}}
  \lesssim_s
  \|f\|_{\A^s_{x,t}}\|g\|_{\A^{-s}_{x,t}}.
 \end{equation}
 \item If $m \geq 0$ and $s>m+1$, then
 \begin{equation}\label{eq:derivative-L1-to-A}
  \|\partial_x^m f\|_{\A^{-s}_{x,t}}
  \lesssim_{s,m}
  \|f\|_{C_t^0L_x^1}.
 \end{equation}
 \item Suppose $K\geq1$, $\ell\in\Z\setminus\{0\}$, and
 \[
  \supp\widehat f(t,\cdot)\subset[-|\ell|K/2,|\ell|K/2]
 \]
 for every $t$.  Then
 \begin{equation}\label{eq:modulation-Wiener}
  \|e^{2\pi i\ell Kx}f\|_{\A^{-s}_{x,t}}
  \lesssim_s
  (|\ell|K)^{-s}\|f\|_{\A^0_{x,t}}.
 \end{equation}
\end{enumerate}
\end{lemma}

\begin{proof}
Peetre's inequality gives
\[
 \wt{n+m}^{-s}\lesssim_s\wt n^s\wt m^{-s}.
\]
Inserting this into the convolution formula for $\widehat{fg}$ proves \eqref{eq:Wiener-product}. For \eqref{eq:derivative-L1-to-A}, use
\[
 \|\widehat f(\cdot,n)\|_{C_t^0}
 \leq\|f\|_{C_t^0L_x^1}
\]
together with the summability of 
$
\sum_n\wt n^{m-s}
$
when $s>m+1$. Finally, the Fourier support of $e^{2\pi i\ell Kx}f$ satisfies
\[
 |\ell K+n|\geq|\ell|K/2.
\]
Therefore
\[
 \|e^{2\pi i\ell Kx}f\|_{\A^{-s}_{x,t}}
 =\sum_n\wt{\ell K+n}^{-s}
 \|\widehat f(\cdot,n)\|_{C_t^0}
 \lesssim_s(|\ell|K)^{-s}\|f\|_{\A^0_{x,t}}.
\]
\end{proof}

\subsection{Classical and singular weak solutions} We need to make sense of the solution $u$ with the regularity considered in this paper. To begin with, we state the ordinary weak formulation of \eqref{eq:NLS}.

\begin{definition} \label{def:classical-weak}
Let $u_{\mathrm{initial}}\in L^1(\T)$.  A function
\[
 u\in C([0,1]; L^1(\T))\cap L_{\mathrm{loc}}^3([0,1)\times\T)
\]
with $u(0)=u_{\mathrm{initial}}$ is a classical weak solution of \eqref{eq:NLS} if, for every
$\varphi\in C_c^\infty([0,1)\times\T;\C)$,
\begin{equation}\label{eq:classical-weak-duality}
 \int_0^1\int_{\T}
 \left[u\big(-i\partial_t\overline\varphi
 +\partial_x^2\overline\varphi\big)
 +\sigma|u|^2u\,\overline\varphi\right]\,\dd x\dd t
 =i\int_{\T}u_{\mathrm{initial}}(x)
 \overline{\varphi(0,x)}\,\dd x.
\end{equation}
\end{definition}

For functions rougher than $L^3(\T)$, the pointwise cubic need not be defined.  We therefore encode the cubic interaction before summing the output frequency.

\begin{definition}[\cite{GMPR}]\label{def:absolute-cubic}
Let $s>0$ and $u_1,u_2,u_3\in C_t^0\mathcal D'_x$.  Define
\begin{equation}\label{eq:Cs-def}
 \begin{aligned}
 \CS_s(u_1,u_2,u_3)
 :=\sum_{n_1,n_2,n_3\in\Z}
 \wt{n_1-n_2+n_3}^{-s}
 \left\|
 \widehat u_1(\cdot,n_1)
 \overline{\widehat u_2(\cdot,n_2)}
 \widehat u_3(\cdot,n_3)
 \right\|_{C_t^0}.
 \end{aligned}
\end{equation}
We write $\CS_s(u):=\CS_s(u,u,u)$.
\end{definition}

\begin{definition}\label{def:cubic-product}
Suppose $\CS_s(u)<\infty$ for some $s>0$.  For every $n\in\Z$, define
\begin{equation}\label{eq:N-Fourier}
 \widehat{\N(u)}(t,n)
 :=\sum_{n_1-n_2+n_3=n}
 \widehat u(t,n_1)
 \overline{\widehat u(t,n_2)}
 \widehat u(t,n_3).
\end{equation}
Then the series converges absolutely in $C_t^0$ for each output frequency and
\begin{equation}\label{eq:N-A-bound}
 \|\N(u)\|_{\A^{-s}_{x,t}}\leq\CS_s(u).
\end{equation}
In particular, \eqref{eq:Wiener-Sobolev-embedding} gives
\begin{equation}\label{eq:N-Hminus-bound}
 \|\N(u)\|_{C_t^0H_x^{-s}}
 \leq \|\N(u)\|_{\A^{-s}_{x,t}}
 \leq \CS_s(u),
\end{equation}
so the $H_x^{-s}$--$H_x^s$ duality pairing of $\N(u)$ against a smooth test function is well defined.
\end{definition}

\begin{definition}\label{def:singular-solution}
Fix $s>0$ and let $u_{\mathrm{initial}}\in H_x^{-s}$.  A function
$u\in C([0,1];H_x^{-s})$ is a singular weak solution of \eqref{eq:NLS} with initial datum $u_{\mathrm{initial}}$ if $\CS_s(u)<\infty$ and if
\begin{equation}\label{eq:singular-weak-duality}
 \begin{aligned}
 &\int_0^1
 \left\langle
 u(t),i\partial_t\varphi(t)+\partial_x^2\varphi(t)
 \right\rangle_{H_x^{-s},H_x^s}\,\dd t\\
 &\qquad
 +\sigma\int_0^1
 \left\langle\N(u)(t),\varphi(t)\right\rangle_{H_x^{-s},H_x^s}\,\dd t
 =i\left\langle u_{\mathrm{initial}},\varphi(0)\right\rangle_{H_x^{-s},H_x^s}
 \end{aligned}
\end{equation}
for every $\varphi\in C_c^\infty([0,1)\times\T;\C)$. 
\end{definition}

The terminology is adapted from the absolute Fourier summability framework of \cite{GMPR}.  It is stronger than prescribing one particular approximation of the cubic term.  The following compatibility statement is the only comparison with the ordinary cubic needed for the sharpness assertion.

\begin{lemma}\label{lem:compatibility}
Suppose $u\in C_t^0L_x^3$ and $\CS_s(u)<\infty$ for some $s>0$.  Then
\begin{equation}\label{eq:compatibility}
 \N(u)=|u|^2u
\end{equation}
in $C_t^0\mathcal D'(\T)$.  Consequently, every singular weak solution that also belongs to $C_t^0L_x^3$ is a classical weak solution in the sense of \cref{def:classical-weak}.
\end{lemma}

\begin{proof}
Let $P_N=\mathbb P_{\leq N}$ be a standard smooth Fourier truncation, with symbol $m_N$.  By the triangle inequality in output frequency,
\begin{align*}
 &\big\||P_Nu|^2P_Nu-\N(u)\big\|_{\A^{-s}_{x,t}}
 \leq  \sum_{n_1,n_2,n_3\in\Z}
 \wt{n_1-n_2+n_3}^{-s}
 \big|m_N(n_1)\overline{m_N(n_2)}m_N(n_3)-1\big|\\
 &\hspace{3.2cm}\times
 \left\|
 \widehat u(\cdot,n_1)
 \overline{\widehat u(\cdot,n_2)}
 \widehat u(\cdot,n_3)
 \right\|_{C_t^0}.
\end{align*}
The multiplier factor converges pointwise to zero and is uniformly bounded.  Since $\CS_s(u)<\infty$, dominated convergence yields
\begin{equation}\label{eq:standard-cutoff-to-N}
 |P_Nu|^2P_Nu\longrightarrow\N(u)
 \qquad\text{in }\A^{-s}_{x,t}.
\end{equation}

The operators $P_N$ are uniformly bounded on $L^3(\T)$ and converge strongly to the identity.  Since the image of the continuous map $t\mapsto u(t)$ is compact in $L^3(\T)$, the convergence is uniform in time, and therefore
\[
 P_Nu\longrightarrow u
 \qquad\text{in }C_t^0L_x^3.
\]
Using
\[
 \big\||f|^2f-|g|^2g\big\|_{L^1}
 \lesssim
 \big(\|f\|_{L^3}^2+\|g\|_{L^3}^2\big)\|f-g\|_{L^3},
\]
we obtain
\[
 |P_Nu|^2P_Nu\longrightarrow|u|^2u
 \qquad\text{in }C_t^0L_x^1.
\]
Comparing this limit with \eqref{eq:standard-cutoff-to-N} gives \eqref{eq:compatibility} in $C_t^0\mathcal D'(\T)$.  Substituting \eqref{eq:compatibility} into \eqref{eq:singular-weak-duality} gives \eqref{eq:classical-weak-duality}.
\end{proof}

For completeness, we also record that the signed absolute-Fourier product is robust under arbitrary admissible Fourier cutoffs in the sense used by Christ and Guo--Kwon--Oh \cite{Christ,GKO}.

\begin{proposition}\label{prop:cutoff-independence}
Suppose $\CS_s(u)<\infty$ for some $s>0$, and let $(P_N)$ be any sequence of Fourier cutoff operators satisfying \eqref{eq:Fourier-cutoff}--\eqref{eq:cutoff-uniform}.  Then
\begin{equation}\label{eq:cutoff-convergence}
 |P_Nu|^2P_Nu\longrightarrow\N(u)
 \qquad\text{in }\A^{-s}_{x,t}.
\end{equation}
In particular, the limit exists and is independent of the cutoff sequence.
\end{proposition}

\begin{proof}
As in the argument leading to \eqref{eq:standard-cutoff-to-N}, note that for each triple $(n_1,n_2,n_3)$, the multiplier attached to the truncated cubic is
\[
 m_N(n_1)\overline{m_N(n_2)}m_N(n_3),
\]
which converges to one and is uniformly bounded.  After multiplication by the output weight, the corresponding summand is dominated by a fixed constant times the summable series in \eqref{eq:Cs-def}.  Dominated convergence proves \eqref{eq:cutoff-convergence}.
\end{proof}

\section{The main inductive proposition and proof of the theorem}\label{sec:main-prop}

Fix throughout
\begin{equation}\label{eq:s-beta}
 s>3,
 \qquad
 0<\beta<\frac13.
\end{equation}
For $q\geq0$, we let
\begin{equation}\label{eq:diagonal-parameters}
 \eps_{q+1}=2^{-(q+1)},
 \qquad
 \alpha_{q+1}=\frac16-\eps_{q+1},
 \qquad
 p_{q+1}=3-\eps_{q+1},
\end{equation}
and
\begin{equation}\label{eq:budgets}
 \delta_q=2^{-q-200},
 \qquad
 \gamma_{q+1}=2^{-q-30}.
\end{equation}

The convex integration scheme is based on the relaxed equation
\begin{equation}\label{eq:relaxed-main-prop}
 i\partial_tu_q+\partial_x^2u_q+\sigma|u_q|^2u_q=E_q.
\end{equation}
We say a pair $(u_q,E_q)$ solves \eqref{eq:relaxed-main-prop} at stage $q$ when this identity holds.

\begin{proposition}[Main inductive proposition]\label{prop:main-inductive}
There exist a family of smooth functions $\{(u_q,E_q)\}_{q \geq 0}$ on $[0,1]\times\T$, integers $M_q\geq1$, compact intervals $I_q\Subset(0,1)$, and, for every $q\geq0$, a dyadic number $\lambda_{q+1}$ and an increment
\[
 w_{q+1}:=u_{q+1}-u_q
\]
such that the following properties hold.

\begin{enumerate}[label=\textup{(\roman*)}]
\item The pair $(u_q,E_q)$ solves the equation \eqref{eq:relaxed-main-prop} at stage $q$.

\item The time intervals are given explicitly by
\begin{equation}\label{eq:Iq-explicit}
 I_q
 :=\left[
 \frac12-\frac1{16}-\sum_{j=1}^{q}\lambda_j^{-\beta},
 \frac12+\frac1{16}+\sum_{j=1}^{q}\lambda_j^{-\beta}
 \right],
\end{equation}
where the sum is understood to be empty when $q=0$.  One has
\[
 \supp_tu_q\cup\supp_tE_q\subset I_q,
\]
and
\begin{equation}\label{eq:time-sum}
 \sum_{j=1}^{\infty}\lambda_j^{-\beta}<\frac1{16}.
\end{equation}
Consequently,
\[
 I_q\subset\left(\frac38,\frac58\right)
 \Subset\left(\frac14,\frac34\right)
 \qquad\text{for every }q\geq0.
\]

\item The exponent of $\lambda_{q+1}$ as a power of two is divisible by $2^{q+1}$.  Hence
\begin{equation}\label{eq:mu-nu-K}
 \mu_{q+1}=\lambda_{q+1}^{\eps_{q+1}},
 \qquad
 \nu_{q+1}=\lambda_{q+1}^{1+\eps_{q+1}},
 \qquad
 K_{q+1}=16\nu_{q+1}
\end{equation}
are positive integers.  Moreover,
\begin{equation}\label{eq:old-new-separation-main}
 \mu_{q+1}>100M_q,
 \qquad
 M_{q+1}=5K_{q+1},
\end{equation}
and
\begin{equation}\label{eq:uq-Eq-support}
 \supp\widehat{u_q}(t,\cdot)\cup\supp\widehat{E_q}(t,\cdot)
 \subset[-M_q,M_q]
\end{equation}
for every $t$.  The increment satisfies
\begin{equation}\label{eq:exact-block}
 \begin{aligned}
 \supp\widehat w_{q+1}(t,\cdot)
 \subset B_{q+1}:={}&
 \{K_{q+1}+\eta:\ \eta\in\mu_{q+1}\Z,\ |\eta|\leq2\nu_{q+1}\}\\
 &\cup
 \{2K_{q+1}+\eta+\xi:\ \eta\in\mu_{q+1}\Z,\ |\eta|\leq2\nu_{q+1},\ |\xi|\leq M_q\}.
 \end{aligned}
\end{equation}
In particular,
\begin{equation}\label{eq:block-annulus}
 B_{q+1}\subset\{n\in\Z:2M_q<n<M_{q+1}/2\},
\end{equation}
so the blocks $B_{q+1}$ are pairwise disjoint.

\item The perturbation satisfies the following smallness estimates:
\begin{equation}\label{eq:small-increments-main}
 \|w_{q+1}\|_{C_t^0H_x^{\alpha_{q+1}}}
 +\|w_{q+1}\|_{C_t^0L_x^{p_{q+1}}}
 <2^{-q-10},
\end{equation}
and, with $m_0:=\|u_0\|_{C_t^0L_x^1}>0$,
\begin{equation}\label{eq:L1-budget-main}
 \|w_{q+1}\|_{C_t^0L_x^1}<2^{-q-3}m_0.
\end{equation}

\item The error $E_q$ is bounded by
\begin{equation}\label{eq:error-budget-main}
 \|E_q\|_{\A^{-s}_{x,t}}\leq\delta_q.
\end{equation}

\item The nonlinearity satisfies the following Fourier bound:
\begin{equation}\label{eq:CS-induction-main}
 \CS_s(u_q)
 \leq
 \CS_s(u_0)
 +\sum_{j=0}^{q-1}(\delta_j+\gamma_{j+1}).
\end{equation}

\item The following nontriviality condition holds at each stage: 
\begin{equation}\label{eq:nontriviality-main}
 \|u_q-u_0\|_{C_t^0L_x^1}<\frac12m_0.
\end{equation}
\end{enumerate}
\end{proposition}

\begin{remark}\label{rem:fixed-epsilon}
For one fixed pair $\alpha<1/6$ and $p<3$, one may keep a sufficiently small dyadic rational intermittency exponent $\eps>0$ fixed, provided
\[
 \alpha<\frac{1-\eps}{6(1+\eps)}.
\]
This gives a simpler construction in $C_t^0H_x^\alpha\cap C_t^0L_x^p$. The diagonal choice $\eps_q\downarrow0$ is needed only to obtain the full intersections in \eqref{eq:main-regularity} from a single solution.
\end{remark}

\subsection{Proof of the main theorem from the inductive proposition}

\begin{proof}
Fix $\alpha<1/6$.  Since $\alpha_{q+1}\uparrow1/6$, there exists $q_\alpha$ such that $\alpha<\alpha_{q+1}$ for every $q\geq q_\alpha$.  The monotonicity of the inhomogeneous Sobolev norms and \eqref{eq:small-increments-main} give
\[
 \sum_{q\geq q_\alpha}
 \|w_{q+1}\|_{C_t^0H_x^\alpha}
 \leq
 \sum_{q\geq q_\alpha}
 \|w_{q+1}\|_{C_t^0H_x^{\alpha_{q+1}}}
 <\infty.
\]
The first finitely many increments are smooth, so they belong to $C_t^0H_x^\alpha$ for every $\alpha$.  Hence $u_q$ converges in $C_t^0H_x^\alpha$.

The $L^p$ argument is identical.  Fix $1\leq p<3$.  For all sufficiently large $q$, one has $p<p_{q+1}$.  Since
\[
 \|f\|_{L^p}\leq\|f\|_{L^{p_{q+1}}},
\]
$u_q$ converges in $C_t^0L_x^p$.  We therefore obtain a limit
\begin{equation}\label{eq:limit-regularity}
 u\in\bigcap_{\alpha<1/6}C_t^0H_x^\alpha
 \cap
 \bigcap_{1\leq p<3}C_t^0L_x^p.
\end{equation}

By the explicit formula \eqref{eq:Iq-explicit} and \eqref{eq:time-sum}, every $I_q$ is contained in $(3/8,5/8)$.  Hence all $u_q$ and $E_q$ are supported in one fixed compact subinterval of $(0,1)$.  Therefore the same is true for $u$, and $u(0,\cdot)=0$.  Moreover, \eqref{eq:L1-budget-main} implies
\[
 \|u-u_0\|_{C_t^0L_x^1}
 \leq\sum_{q=0}^{\infty}2^{-q-3}m_0
 =\frac14m_0.
\]
Consequently,
\[
 \|u\|_{C_t^0L_x^1}\geq\frac34m_0>0,
\]
so $u$ is nontrivial.

We next prove that the limiting cubic product is well-defined.  The block relation \eqref{eq:block-annulus} implies that no later increment changes a Fourier coefficient already present in $u_q$.  More precisely, let
\[
 F_q:=\supp\widehat u_0\cup B_1\cup\cdots\cup B_q.
\]
Then
\begin{equation}\label{eq:block-truncation}
 \widehat u_q(t,n)=\one_{F_q}(n)\widehat u(t,n).
\end{equation}
Thus $\CS_s(u_q)$ is the partial sum of the nonnegative series $\CS_s(u)$ over triples whose three input frequencies lie in the first $q$ blocks.  These partial sums are increasing.  By \eqref{eq:CS-induction-main} and the summability of $\delta_j$ and $\gamma_j$, monotone convergence yields
\begin{equation}\label{eq:CS-limit}
 \CS_s(u)<\infty.
\end{equation}
In particular, $\N(u)$ is well-defined by \cref{def:cubic-product}. Furthermore, using the triangle inequality in each output frequency,
\begin{align}
 \|\N(u)-\N(u_q)\|_{\A^{-s}_{x,t}}
 \leq
 \sum_{\substack{n_1,n_2,n_3\in\Z\\
 \{n_1,n_2,n_3\}\not\subset F_q}}
 &\wt{n_1-n_2+n_3}^{-s} \notag \\
 &\times
 \left\|
 \widehat u(\cdot,n_1)
 \overline{\widehat u(\cdot,n_2)}
 \widehat u(\cdot,n_3)
 \right\|_{C_t^0} \longrightarrow 0 . \label{eq:N-convergence}
\end{align}

The right-hand side is a decreasing tail of the absolutely convergent
series defining $\CS_s(u)$ and therefore tends to zero. This proves the convergence of the singular cubic product. 

Now we show that the limit $u$ is a singular weak solution. For each $q$, the classical cubic product agrees with $\N(u_q)$ because $u_q$ is a trigonometric polynomial.  Testing \eqref{eq:relaxed-main-prop} against a compactly supported smooth function and using the convergence $u_q\to u$ in $C_t^0L_x^1$, we pass the two linear terms to the limit in distributions.  The nonlinear term converges by \eqref{eq:N-convergence}, and $E_q\to0$ in $\A^{-s}_{x,t}$ by \eqref{eq:error-budget-main}.  Therefore
\[
 i\partial_tu+\partial_x^2u+\sigma\N(u)=0
\]
in distributions on $(0,1)\times\T$.  Since $u$ and $\N(u)$ vanish on a neighborhood of $t=0$, this interior identity extends to every test function in $C_c^\infty([0,1)\times\T)$ with the zero boundary term in \eqref{eq:singular-weak-duality}.  Hence $u$ is a singular weak solution with zero initial datum.

Finally, let $v$ be any singular weak solution with zero initial datum such that $v\in C_t^0H_x^{1/6}$.  The Sobolev embedding $H^{1/6}(\T)\hookrightarrow L^3(\T)$ and \cref{lem:compatibility} identify $\N(v)$ with the ordinary cubic product $|v|^2v$.  Thus $v$ is a classical weak solution in $C_t^0H_x^{1/6}$, and the unconditional uniqueness theorem of Guo, Kwon, and Oh \cite{GKO} implies $v\equiv0$.  This proves the endpoint assertion and completes the proof.
\end{proof}

\section{Intermittent slabs and carrier-wave estimates}\label{sec:slabs}

The analytic building block is the intermittent slab introduced for periodic gKdV in \cite{GMPR}; see also \cite{GismondiSlabs}. We recall its construction and properties with a self-contained proof in this section. In addition, we provide estimates for placing the intermittent slab in the carrier-wave perturbation.  All constants below are uniform in the large parameter $\lambda$ after $\eps$ has been fixed.

\subsection{Construction and properties of the intermittent slab}\label{subsec:Theta}

Choose a nonzero real function $\psi\in C_c^\infty((-1/4,1/4))$ with
\[
 \int_{\R}\psi(x)\,\dd x=0.
\]
We use the real-line Fourier transform
$\widehat f(\xi)=\int_{\R}f(x)e^{-2\pi ix\xi}\,\dd x$.
Set $\widetilde\psi(x)=\psi(-x)$ and $\phi_0=\psi*\widetilde\psi$.
Then $\phi_0$ is real and even, and
\[
 \widehat{\phi_0}(\xi)=|\widehat\psi(\xi)|^2\geq0,
 \qquad
 \int_{\R}\phi_0=0.
\]
Moreover,
\begin{equation}\label{eq:phi-cubic-positive}
 \int_{\R}\phi_0^3\,\dd x>0.
\end{equation}
Indeed, by Fourier inversion,
\[
 \int_{\R}\phi_0^3\,\dd x
 =\iint_{\R^2}
 \widehat{\phi_0}(\xi)
 \widehat{\phi_0}(\eta)
 \widehat{\phi_0}(-\xi-\eta)
 \,\dd\xi\dd\eta.
\]
The integrand is nonnegative.  Since $\widehat{\phi_0}$ is a nonzero
real-analytic function, its zero set on $\R$ has empty interior.  Hence
$\{\widehat{\phi_0}>0\}$ is open and dense.  The three corresponding
open dense subsets of $\R^2$ determined by $\xi$, $\eta$, and
$-\xi-\eta$ have nonempty intersection, so the integrand is positive on
an open set of pairs.  This proves \eqref{eq:phi-cubic-positive}.  After normalization, we henceforth assume
\begin{equation}\label{eq:phi-properties}
 \phi\in C_c^\infty((-1/2,1/2)),
 \qquad
 \phi\text{ real and even},
 \qquad
 \int_{\R}\phi=0,
 \qquad
 \int_{\R}\phi^3=1,
 \qquad
 \widehat\phi\geq0.
\end{equation}

Fix $0<\eps\leq1/2$, and let $\lambda$ be a sufficiently large power of two such that
\begin{equation}\label{eq:integrality-slab}
 \mu:=\lambda^\eps\in \mathbb N,
 \qquad
 \lambda^{1-\eps}\in \mathbb N.
\end{equation}
Define the raw pulse train
\begin{equation}\label{eq:rho-def}
 \rho_{\lambda,\eps}(x)
 :=\sum_{m\in\Z}
 \lambda^{(1-\eps)/3}
 \phi(\lambda x+\lambda^{1-\eps}m).
\end{equation}
Each pulse has width $O(\lambda^{-1})$, whereas neighboring pulse centers are separated by $\mu^{-1}=\lambda^{-\eps}$.  Hence the supports of the pulses are disjoint for all sufficiently large $\lambda$. The next lemma provides some useful properties of $\rho_{\lambda,\epsilon}$.

\begin{lemma}\label{lem:raw-slab}
The function $\rho :=\rho_{\lambda,\eps}$ is real, even, and $\mu^{-1}$-periodic.  It satisfies
\begin{equation}\label{eq:raw-moments}
 \int_{\T}\rho\,\dd x=0,
 \qquad
 \int_{\T}\rho^3\,\dd x=1.
\end{equation}
For every $1\leq p<\infty$,
\begin{equation}\label{eq:raw-Lp-exact}
 \|\rho\|_{L^p(\T)}^p
 =\lambda^{(1-\eps)(p/3-1)}\|\phi\|_{L^p(\R)}^p,
\end{equation}
and
\begin{equation}\label{eq:raw-Linf}
 \|\rho\|_{L^\infty(\T)}
 \lesssim\lambda^{(1-\eps)/3}.
\end{equation}
Its Fourier coefficients are supported on $\mu\Z$ and are given by
\begin{equation}\label{eq:raw-Fourier}
 \widehat\rho(\mu k)
 =\lambda^{-\frac23(1-\eps)}
 \widehat\phi(\lambda^{\eps-1}k)\geq0,
 \qquad k\in\Z.
\end{equation}
More generally, for $n=2,3$,
\begin{equation}\label{eq:rhoj-Fourier}
 \widehat{\rho^n}(\mu k)
 =\lambda^{(1-\eps)(\frac{n}{3}-1)}
 \widehat{\phi^n}(\lambda^{\eps-1}k).
\end{equation}
All other Fourier coefficients in \eqref{eq:raw-Fourier} and \eqref{eq:rhoj-Fourier} vanish.
\end{lemma}

\begin{proof}
The periodicity follows by translating $x$ by $\mu^{-1}$ and reindexing $m$.  Evenness follows from the evenness of $\phi$ and the change $m\mapsto-m$.  There are exactly $\mu$ disjoint pulses on one period of $\T$.  Therefore, for every integer $n\geq1$,
\begin{equation}\label{eq:raw-j-moment}
 \int_{\T}\rho^n\,\dd x
 =\lambda^{n(1-\eps)/3}\mu\lambda^{-1}
 \int_{\R}\phi^n\,\dd x
 =\lambda^{(1-\eps)(n/3-1)}
 \int_{\R}\phi^n\,\dd x.
\end{equation}
The cases $n=1$ and $n=3$ give \eqref{eq:raw-moments}, and the same calculation with $|\phi|^p$ gives \eqref{eq:raw-Lp-exact}.  The $L^\infty$ bound is immediate from the disjointness of the pulses. Next, Poisson summation applied to \eqref{eq:rho-def} gives
\[
 \rho(x)
 =\sum_{k\in\Z}
 \lambda^{-\frac23(1-\eps)}
 \widehat\phi(\lambda^{\eps-1}k)
 e^{2\pi i\mu kx},
\]
which proves \eqref{eq:raw-Fourier}.  Applying the same calculation to the disjoint pulse representation of $\rho^j$ gives \eqref{eq:rhoj-Fourier}.
\end{proof}

We now impose an exact finite Fourier support.  Let
\begin{equation}\label{eq:nu-def}
 \nu=\lambda^{1+\eps}=\lambda\mu.
\end{equation}
Choose an even function $\chi\in C_c^\infty(\R)$ such that
\begin{equation}\label{eq:chi-properties}
 0\leq\chi\leq1,
 \qquad
 \chi(\xi)=1\text{ for }|\xi|\leq1,
 \qquad
 \chi(\xi)=0\text{ for }|\xi|\geq2.
\end{equation}
Define
\begin{equation}\label{eq:rho-sharp}
 \widehat{\rho^\sharp}(n)
 :=\chi(n/\nu)\widehat\rho(n).
\end{equation}

\begin{lemma}[Fourier cutoff error]\label{lem:cutoff-error}
Fix $0<\eps\leq1/2$.  For every $m,N\geq0$, there exists $C_{m,N,\eps}$ independent of $\lambda$ such that
\begin{equation}\label{eq:cutoff-error}
 \|\rho-\rho^\sharp\|_{C^m(\T)}
 \leq C_{m,N,\eps}\lambda^{-N}.
\end{equation}
Consequently, for every $N\geq0$,
\begin{equation}\label{eq:cubic-cutoff-moment}
 \int_{\T}(\rho^\sharp)^3\,\dd x
 =1+O_{N,\eps}(\lambda^{-N}).
\end{equation}
\end{lemma}

\begin{proof}
The multiplier in \eqref{eq:rho-sharp} differs from one only when $|n|>\nu$.  Since $n=\mu k$ on the support of $\widehat\rho$, this requires $|k|>\lambda$.  By \eqref{eq:raw-Fourier},
\[
 \|\rho-\rho^\sharp\|_{C^m}
 \lesssim_m
 \lambda^{-\frac23(1-\eps)}
 \sum_{|k|>\lambda}
 (\mu|k|)^m
 \left|\widehat\phi(\lambda^{\eps-1}k)\right|.
\]
For any $R>0$, the Schwartz decay of $\widehat\phi$ gives
\[
 \left|\widehat\phi(\lambda^{\eps-1}k)\right|
 \leq C_R(1+\lambda^{\eps-1}|k|)^{-R}.
\]
For $|k|>\lambda$, the factor in parentheses is at least
$\lambda^{\eps-1}|k|$.  Taking $R>m+1$ and summing the resulting power
tail gives
\[
 \|\rho-\rho^\sharp\|_{C^m}
 \lesssim_{m,R}
 \lambda^{-\frac23(1-\eps)+\eps m+(1-\eps)R}
 \lambda^{m-R+1}
 =\lambda^{m(1+\eps)+1-\frac23(1-\eps)-\eps R}.
\]
Choosing $R$ sufficiently large in terms of $m,N,\eps$ makes this at
most $C_{m,N,\eps}\lambda^{-N}$.  This proves
\eqref{eq:cutoff-error}.

To prove \eqref{eq:cubic-cutoff-moment}, let
\[
 d:=\rho^\sharp-\rho.
\]
Then
\[
 (\rho^\sharp)^3-\rho^3
 =3\rho^2d+3\rho d^2+d^3.
\]
By H\"older's inequality,
\[
 \left|
 \int_{\T}\big((\rho^\sharp)^3-\rho^3\big)\,\dd x
 \right|
 \leq
 3\|d\|_{L^\infty}\|\rho\|_{L^2}^2
 +3\|d\|_{L^\infty}^2\|\rho\|_{L^1}
 +\|d\|_{L^\infty}^3.
\]
Using \eqref{eq:raw-Lp-exact} with $p=2$ and $p=1$, respectively,
\[
 \|\rho\|_{L^2}^2
 \lesssim_\eps
 \lambda^{-\frac13(1-\eps)},
 \qquad
 \|\rho\|_{L^1}
 \lesssim_\eps
 \lambda^{-\frac23(1-\eps)}.
\]
On the other hand, \eqref{eq:cutoff-error} with $m=0$ gives
\[
 \|d\|_{L^\infty}
 \lesssim_{N,\eps}\lambda^{-N}.
\]
Consequently, for every $N\geq0$,
\[
 \left|
 \int_{\T}\big((\rho^\sharp)^3-\rho^3\big)\,\dd x
 \right|
 \lesssim_{N,\eps}
 \lambda^{-N-\frac13(1-\eps)}
 +\lambda^{-2N-\frac23(1-\eps)}
 +\lambda^{-3N}
 \lesssim_{N,\eps}\lambda^{-N}.
\]
Combining this with \eqref{eq:raw-moments} proves
\eqref{eq:cubic-cutoff-moment}.
\end{proof}

For large $\lambda$, the quantity in \eqref{eq:cubic-cutoff-moment} is positive.  Define
\begin{equation}\label{eq:Theta-def}
 c_{\lambda,\eps}
 :=\left(\int_{\T}(\rho^\sharp)^3\,\dd x\right)^{-1/3},
 \qquad
 \Theta_{\lambda,\eps}:=c_{\lambda,\eps}\rho^\sharp.
\end{equation}
By \cref{lem:cutoff-error},
\begin{equation}\label{eq:c-close}
 c_{\lambda,\eps}=1+O_{N,\eps}(\lambda^{-N})
\end{equation}
for every $N$, and in particular
\begin{equation}\label{eq:c-uniform}
 \frac12\leq c_{\lambda,\eps}\leq2
\end{equation}
for all sufficiently large $\lambda$. The truncated normalized intermittent slab $\Theta_{\lambda,\eps}$ is the one we will use in our perturbation. The next proposition provides the key properties we need for the subsequent estimates, most of which are derived from the original slab $\rho_{\lambda,\eps}$.

\begin{proposition}\label{prop:Theta}
Let $\Theta=\Theta_{\lambda,\eps}$ be defined by \eqref{eq:Theta-def}.  For all sufficiently large $\lambda$, the following properties hold.

\begin{enumerate}[label=\textup{(\roman*)}]
 \item $\Theta$ is real and even, and
 \begin{equation}\label{eq:Theta-Fourier-support}
  \widehat\Theta(0)=0,
  \qquad
  \widehat\Theta(n)\geq0,
  \qquad
  \supp\widehat\Theta
  \subset \mu\Z\cap[-2\nu,2\nu].
 \end{equation}
 \item The cubic moment is exact:
 \begin{equation}\label{eq:Theta-cubic}
  \int_{\T}\Theta^3\,\dd x=1.
 \end{equation}
 \item For every $1\leq p\leq\infty$,
 \begin{equation}\label{eq:Theta-Lp}
  \|\Theta\|_{L^p(\T)}
  \lesssim_p
  \lambda^{(1-\eps)(1/3-1/p)}.
 \end{equation}
 In particular,
 \begin{equation}\label{eq:Theta-L1-L2}
  \|\Theta\|_{L^1}
  \lesssim \lambda^{-\frac23(1-\eps)},
  \qquad
  \int_{\T}\Theta^2\,\dd x
  \lesssim \lambda^{-\frac13(1-\eps)}.
 \end{equation}
 \item The absolute Fourier masses satisfy
 \begin{equation}\label{eq:Theta-A0}
  \|\Theta\|_{\A_x^0} := \|\widehat{\Theta}\|_{\ell^1}
  \lesssim\lambda^{\frac13(1-\eps)},
  \qquad
  \|\Theta^j\|_{\A_x^0} := \|\widehat{\Theta^j}\|_{\ell^1}
  \lesssim_j\lambda^{\frac j3(1-\eps)},
  \quad j=2,3.
 \end{equation}
 \item For $j=2,3$ and every $s>1$,
 \begin{equation}\label{eq:Theta-Pneq}
  \|\Pnz(\Theta^j)\|_{\A_x^{-s}}
  \lesssim_{s,j}
  \lambda^{(1-\eps)(j/3-1)}\mu^{-s}.
 \end{equation}
 \item The Fourier coefficients of $\Theta$ are real and even. The zero modes of the nonlinear interactions satisfy
 \begin{equation}\label{eq:Theta-Q2}
  \sum_{\eta_1-\eta_2=0}
  \widehat\Theta(\eta_1)\widehat\Theta(\eta_2)
  =\int_{\T}\Theta^2\,\dd x
  \lesssim\lambda^{-\frac13(1-\eps)},
 \end{equation}
 and
 \begin{equation}\label{eq:Theta-Q3}
  \sum_{\eta_1-\eta_2+\eta_3=0}
  \prod_{r=1}^3\widehat\Theta(\eta_r)
  =\int_{\T}\Theta^3\,\dd x=1.
 \end{equation}
 Because every Fourier coefficient is nonnegative, the same identities hold after absolute values are inserted term by term.
\end{enumerate}
\end{proposition}

\begin{proof}
All the properties in \eqref{eq:Theta-Fourier-support} follow from \eqref{eq:raw-Fourier} and the even nonnegative cutoff in \eqref{eq:chi-properties}.  The identity \eqref{eq:Theta-cubic} is the definition of the normalizing constant.

The raw estimate \eqref{eq:raw-Lp-exact}, the rapid cutoff error
\eqref{eq:cutoff-error}, and the uniform bound \eqref{eq:c-uniform}
imply \eqref{eq:Theta-Lp}.  More explicitly, for finite $p$,
$\|\rho^\sharp\|_{L^p}\leq\|\rho\|_{L^p}
+\|\rho-\rho^\sharp\|_{C^0}$.  If $p<3$, choose the decay order $N>(1-\eps)(1/p-1/3)$; if
$p\geq3$, the cutoff error is bounded by the nonnegative target power.  The $p=\infty$ case is identical using
\eqref{eq:raw-Linf}.  The two estimates in \eqref{eq:Theta-L1-L2} are
the cases $p=1$ and $p=2$.

Since all Fourier coefficients are nonnegative,
\[
 \|\Theta\|_{\A_x^0}
 =\sum_n\widehat\Theta(n)
 =\Theta (0)
 \lesssim \lambda^{(1-\eps)/3},
\]
where we used \eqref{eq:Theta-Lp}.  This proves the first estimate in \eqref{eq:Theta-A0}.  A direct convolution estimate then gives the bounds in $\A^0_x$ for $\Theta^2$ and $\Theta^3$.

We now prove \eqref{eq:Theta-Pneq}, including the effect of the exact cutoff and normalization.  Since
\[
 0\leq\widehat{\rho^\sharp}(n)\leq\widehat\rho(n)
\]
for every $n$, and since the coefficients of $\rho^j$ are the
$j$-fold convolution of the nonnegative sequence $\widehat\rho$,
convolution of nonnegative sequences yields the termwise domination
\begin{equation}\label{eq:termwise-domination}
 0\leq
 \widehat{(\rho^\sharp)^j}(n)
 \leq
 \widehat{\rho^j}(n),
 \qquad j=2,3.
\end{equation}
Using \eqref{eq:c-uniform}, \eqref{eq:rhoj-Fourier}, and the bound
$|\widehat{\phi^j}(\xi)|\leq\|\phi^j\|_{L^1}$, we find
\[
 \begin{aligned}
 \|\Pnz(\Theta^j)\|_{\A_x^{-s}}
 &\leq C_j
 \sum_{k\neq0}\wt{\mu k}^{-s}
 \widehat{\rho^j}(\mu k)\\
 &\leq C_j\lambda^{(1-\eps)(j/3-1)}
 \sum_{k\neq0}\wt{\mu k}^{-s}
 |\widehat{\phi^j}(\lambda^{\eps-1}k)|\\
 &\lesssim_{s,j}
 \lambda^{(1-\eps)(j/3-1)}\mu^{-s}
 \sum_{k\neq0}|k|^{-s}.
 \end{aligned}
\]
The last series converges for $s>1$, proving \eqref{eq:Theta-Pneq} with a constant independent of $\lambda$.

Finally, since $\Theta$ is real and even, its Fourier coefficients are
real and even.  Thus
\[
 \sum_{\eta_1-\eta_2=0}
 \widehat\Theta(\eta_1)\widehat\Theta(\eta_2)
 =\sum_\eta|\widehat\Theta(\eta)|^2
 =\int_{\T}\Theta^2\,\dd x.
\]
Likewise, after the change $\eta_2\mapsto-\eta_2$, the left side of
\eqref{eq:Theta-Q3} is the zero Fourier coefficient of $\Theta^3$.  This proves \eqref{eq:Theta-Q2} and \eqref{eq:Theta-Q3}.
\end{proof}

\subsection{Placing the slab on a carrier wave}

For an integer $\ell\neq0$, a carrier frequency $K\geq1$, an intermittent profile $\Theta:= \Theta(x)$, a time cutoff function $h:=h(t)$, and an amplitude $c:= c(t,x)$, define
\begin{equation}\label{eq:carrier-operator}
 \mathscr W_{\ell,K}[\Theta,h,c]
 :=h(t)\Theta(x)c(t,x)e^{2\pi i\ell Kx}.
\end{equation}
The following lemma collects the general facts about carrier waves used in the iteration.

\begin{lemma}\label{lem:carrier-slab}
Let $\Theta=\Theta_{\lambda,\eps}$ be given by \cref{prop:Theta}, and suppose
\begin{equation}\label{eq:carrier-amplitude-support}
 \supp\widehat c(t,\cdot)\subset[-M,M]
\end{equation}
for every $t$.  Then:
\begin{enumerate}[label=\textup{(\roman*)}]
 \item the exact Fourier support satisfies
 \begin{equation}\label{eq:carrier-support-general}
  \supp\widehat{\mathscr W_{\ell,K}[\Theta,h,c]}(t,\cdot)
  \subset
  \{\ell K+\eta+\xi:
  \eta\in\mu\Z,\ |\eta|\leq2\nu,\ |\xi|\leq M\};
 \end{equation}
 \item for every $1\leq p\leq\infty$,
 \begin{equation}\label{eq:carrier-Lp-general}
  \|\mathscr W_{\ell,K}[\Theta,h,c]\|_{C_t^0L_x^p}
  \leq
  \|h\|_{L_t^\infty}\|c\|_{C_t^0L_x^\infty}
  \|\Theta\|_{L_x^p};
 \end{equation}
 \item if $2\nu+M\leq|\ell|K/2$, then for every $s\geq0$,
 \begin{equation}\label{eq:carrier-Wiener-general}
  \|\mathscr W_{\ell,K}[\Theta,h,c]\|_{\A^{-s}_{x,t}}
  \lesssim_s
  (|\ell|K)^{-s}
  \|h\Theta c\|_{\A^0_{x,t}}.
 \end{equation}
\end{enumerate}
\end{lemma}

\begin{proof}
The support statement follows from convolution and \eqref{eq:Theta-Fourier-support}.  The $L^p$ estimate is immediate because the carrier has unit modulus.  Under the last support assumption, the unmodulated factor $h\Theta c$ lies in $[-|\ell|K/2,|\ell|K/2]$, so \eqref{eq:carrier-Wiener-general} is exactly \eqref{eq:modulation-Wiener} from \cref{lem:Wiener}.
\end{proof}

\section{Proof of the main inductive proposition}\label{sec:proof-main-prop}

This section is devoted to the proof of \cref{prop:main-inductive}. Firstly, we fix an old-stage solution $(u_q,E_q)$ of the iteration and construct the candidate solution $(u_{q+1},E_{q+1})$ for the next stage. Then we carry out the estimates mentioned in the iterative proposition for $(u_{q+1},E_{q+1})$.

\subsection{The two-carrier perturbation and exact support}\label{subsec:one-step}

Assume that $u_q,E_q\in C^\infty([0,1]\times\T)$ satisfy all the properties stated in \cref{prop:main-inductive}. Then
\[
i\partial_tu_q+\partial_x^2u_q+\sigma|u_q|^2u_q=E_q
\]
and $(u_q, E_q)$ have finite spatial Fourier support,
\[
\operatorname{supp}\widehat u_q(t,\cdot)
\cup\operatorname{supp}\widehat E_q(t,\cdot)
\subset[-M_q,M_q].
\]
In addition, we recall explicitly the inductive residual bound
\[
 \|E_q\|_{\A^{-s}_{x,t}}\leq\delta_q
\]
from \eqref{eq:error-budget-main}.  Since $E_q$ is supported in spatial frequencies $|k|\leq M_q$, \eqref{eq:Wiener-product} in \cref{lem:Wiener} implies
\begin{equation}\label{eq:old-error-positive-Wiener}
 \|E_q\|_{\A^s_{x,t}}
 \leq \langle M_q\rangle^{2s}\|E_q\|_{\A^{-s}_{x,t}}
 \leq \langle M_q\rangle^{2s}\delta_q.
\end{equation}
Recall also that their temporal supports lie in the explicit interval
\begin{equation}\label{eq:Iq-at-stage-q}
 I_q
 =\left[
 \frac12-\frac1{16}-\sum_{j=1}^{q}\lambda_j^{-\beta},
 \frac12+\frac1{16}+\sum_{j=1}^{q}\lambda_j^{-\beta}
 \right].
\end{equation}

Choose a scalar $A_{q+1}$ such that
\begin{equation}\label{eq:A-choice}
 A_{q+1}>1+\|E_q\|_{C_t^0L_x^\infty}
\end{equation}
and define
\begin{equation}\label{eq:a-b-choice}
 a_{q+1}=A_{q+1}^{1/3},
 \qquad
 b_{q+1}(t,x)
 =-\sigma a_{q+1}^{-2}\overline{E_q(t,x)}.
\end{equation}
Then
\begin{equation}\label{eq:b-properties}
 a_{q+1}^2\overline{b_{q+1}}=-\sigma E_q,
 \qquad
 \|b_{q+1}\|_{C_t^0L_x^\infty}\leq a_{q+1},
 \qquad
 \supp\widehat{b_{q+1}}(t,\cdot)\subset[-M_q,M_q].
\end{equation}

Take the intermittency exponent $\eps_{q+1}$ from \eqref{eq:diagonal-parameters} and choose a sufficiently large admissible frequency $\lambda_{q+1}$.  Define
\begin{equation}\label{eq:one-step-frequency-definitions}
 \mu_{q+1}=\lambda_{q+1}^{\eps_{q+1}},
 \qquad
 \nu_{q+1}=\lambda_{q+1}^{1+\eps_{q+1}},
 \qquad
 K_{q+1}=16\nu_{q+1}, 
 \qquad
 M_{q+1}=5K_{q+1}
\end{equation}
in agreement with \eqref{eq:mu-nu-K}, and require
\begin{equation}\label{eq:separation-step}
 \mu_{q+1}>100M_q.
\end{equation}
According to \eqref{eq:Iq-explicit},
\begin{equation}\label{eq:Iqplus1-explicit-step}
 I_{q+1}
 =\left[
 \inf I_q-\lambda_{q+1}^{-\beta},
 \sup I_q+\lambda_{q+1}^{-\beta}
 \right].
\end{equation}
Choose a smooth temporal cutoff $h_{q+1}$ such that
\begin{equation}\label{eq:h-properties}
 0\leq h_{q+1}\leq1,
 \qquad
 h_{q+1}=1\text{ on }I_q,
 \qquad
 \supp h_{q+1}\subset I_{q+1},
 \qquad
 \|\partial_t h_{q+1}\|_{L_t^\infty}
 \lesssim\lambda_{q+1}^{\beta}.
\end{equation}

\begin{lemma}[Two-carrier algebraic identity]\label{lem:polarization}
Let $a>0$, $b\in\C$, and $z\in\C$ with $|z|=1$.  Then
\begin{equation}\label{eq:polarization}
 |az+bz^2|^2(az+bz^2)
 =a^2\overline b
 +a(a^2+2|b|^2)z
 +b(2a^2+|b|^2)z^2
 +ab^2z^3.
\end{equation}
\end{lemma}

\begin{proof}
Since
\[
 |az+bz^2|^2=a^2+|b|^2+a\overline b z^{-1}+abz,
\]
multiplication by $az+bz^2$ gives \eqref{eq:polarization}.
\end{proof}

Define the carrier polynomial
\begin{equation}\label{eq:V-and-w}
 V_{q+1}
 :=a_{q+1}e^{2\pi iK_{q+1}x}+b_{q+1}e^{4\pi iK_{q+1}x}
\end{equation}
and let
\begin{equation}\label{eq:Theta-qplus1}
 \Theta_{q+1}
 :=\Theta_{\lambda_{q+1},\eps_{q+1}}
\end{equation}
be the normalized intermittent slab constructed in \cref{prop:Theta}.  
Using the carrier-wave operator \eqref{eq:carrier-operator}, set
\begin{equation}\label{eq:packet-decomp}
 w_{q+1}^{(1)}
 :=\mathscr W_{1,K_{q+1}}[\Theta_{q+1},h_{q+1},a_{q+1}],
 \qquad
 w_{q+1}^{(2)}
 :=\mathscr W_{2,K_{q+1}}[\Theta_{q+1},h_{q+1},b_{q+1}],
\end{equation}
and
\begin{equation}\label{eq:packets}
 w_{q+1}:=h_{q+1}\Theta_{q+1} V_{q+1} =w_{q+1}^{(1)}+w_{q+1}^{(2)}
 ,
 \qquad
 u_{q+1}:=u_q+w_{q+1}.
\end{equation}
Hence, $\omega_{q+1}$ is the perturbation and $u_{q+1}$ given above is our new candidate solution to \eqref{eq:relaxed-main-prop}. The following lemma gives the Fourier support of the perturbation.
\begin{lemma}\label{lem:packet-support}
For every $t$, with the parameters defined in \eqref{eq:one-step-frequency-definitions}, we have that
\begin{equation}\label{eq:packet-support-1}
 \supp\widehat w_{q+1}^{(1)}(t,\cdot)
 \subset
 \{K_{q+1}+\eta:\eta\in\mu_{q+1}\Z,\ |\eta|\leq2\nu_{q+1}\},
\end{equation}
and
\begin{equation}\label{eq:packet-support-2}
 \supp\widehat w_{q+1}^{(2)}(t,\cdot)
 \subset
 \{2K_{q+1}+\eta+\xi:\eta\in\mu_{q+1}\Z,\ |\eta|\leq2\nu_{q+1},\ |\xi|\leq M_q\}.
\end{equation}
Consequently,
\begin{equation}\label{eq:rough-w-annulus}
 14\nu_{q+1}\leq n\leq35\nu_{q+1}
 \qquad\text{for every }n\in\supp\widehat w_{q+1}(t,\cdot),
\end{equation}
for sufficiently large $\lambda_{q+1}$.  In particular,
\begin{equation}\label{eq:block-step}
 \supp\widehat w_{q+1}(t,\cdot)
 \subset\{2M_q<n<M_{q+1}/2\}.
\end{equation}
\end{lemma}

\begin{proof}
Recall from \eqref{eq:a-b-choice} the definitions of $a_{q+1}$ and $b_{q+1}$. Apply \cref{lem:carrier-slab}, \textup{(i)} first with $\ell=1$ and the constant amplitude $c=a_{q+1}$, and then with $\ell=2$ and $c=b_{q+1}$, whose support is contained in $[-M_q,M_q]$ by \eqref{eq:b-properties}.  This gives \eqref{eq:packet-support-1}--\eqref{eq:packet-support-2}.  The first carrier-wave packet $\omega^{(1)}_{q+1}$ lies in $[K_{q+1}-2\nu_{q+1},K_{q+1}+2\nu_{q+1}]=[14\nu_{q+1},18\nu_{q+1}]$.  The second lies in
\[
 [2K_{q+1}-2\nu_{q+1}-M_q,2K_{q+1}+2\nu_{q+1}+M_q].
\]
Since $M_q<\mu_{q+1}/100<\nu_{q+1}/100$, this interval is contained in $[29\nu_{q+1},35\nu_{q+1}]$.  This proves \eqref{eq:rough-w-annulus}.  The lower bound in \eqref{eq:block-step} follows from $14\nu_{q+1}>2M_q$, and the upper bound follows from $35\nu_{q+1}<40\nu_{q+1}=M_{q+1}/2$.
\end{proof}

In the next section we will establish all estimates for the perturbation as well as the new error. For notational convenience, in the subsequent estimates, all constants that depend only on the earlier stage $q$, on $s,\beta$, on the prescribed exponent $\eps_{q+1}$, and on the fixed profiles will be denoted by $C_q$.  They may depend on finitely many smooth and Wiener norms of $u_q,E_q,b_{q+1}$, and on $A_{q+1}$, but they are independent of the new large parameter $\lambda_{q+1}$. 

\subsection{Size of the perturbation}

\begin{lemma}[$L^p$ and Sobolev size]\label{lem:w-size}
For every $1\leq p\leq\infty$,
\begin{equation}\label{eq:w-Lp}
 \|w_{q+1}\|_{C_t^0L_x^p}
 \leq C_q
 \lambda_{q+1}^{(1-\eps_{q+1})(1/3-1/p)}.
\end{equation}
In particular,
\begin{equation}\label{eq:w-L1}
 \|w_{q+1}\|_{C_t^0L_x^1}
 \leq C_q\lambda_{q+1}^{-\frac23(1-\eps_{q+1})},
\end{equation}
and
\begin{equation}\label{eq:w-L2}
 \|w_{q+1}\|_{C_t^0L_x^2}^2
 \leq C_q\lambda_{q+1}^{-\frac13(1-\eps_{q+1})}.
\end{equation}
For every $\alpha\in\R$, we also have
\begin{equation}\label{eq:w-Halpha}
 \|w_{q+1}\|_{C_t^0H_x^\alpha}
 \leq C_{q,\alpha}
 \lambda_{q+1}^{(1+\eps_{q+1})\alpha-\frac16(1-\eps_{q+1})}.
\end{equation}
\end{lemma}

\begin{proof}
By \eqref{eq:b-properties},
\[
 |V_{q+1}(t,x)|\leq a_{q+1}+|b_{q+1}(t,x)|\leq2a_{q+1}.
\]
Thus \eqref{eq:w-Lp} follows from \eqref{eq:Theta-Lp}.  The estimates \eqref{eq:w-L1} and \eqref{eq:w-L2} are the cases $p=1$ and $p=2$.

By \cref{lem:packet-support}, every Fourier frequency of $w_{q+1}$ is comparable to $K_{q+1}\simeq\nu_{q+1}=\lambda_{q+1}^{1+\eps_{q+1}}$.  Plancherel therefore gives, for every real $\alpha$,
\[
 \|w_{q+1}\|_{H^\alpha}
 \lesssim_\alpha K_{q+1}^\alpha\|w_{q+1}\|_{L^2}.
\]
Combining this with \eqref{eq:w-L2} proves \eqref{eq:w-Halpha}.
\end{proof}

Note that for the Sobolev exponent $\alpha_{q+1}=1/6-\eps_{q+1}$, the power in \eqref{eq:w-Halpha} is
\begin{equation}\label{eq:H-diagonal-exponent}
 (1+\eps_{q+1})\left(\frac16-\eps_{q+1}\right)
 -\frac16(1-\eps_{q+1})
 =-\frac23\eps_{q+1}-\eps_{q+1}^2<0.
\end{equation}
For $p_{q+1}=3-\eps_{q+1}$, the power in \eqref{eq:w-Lp} is
\begin{equation}\label{eq:Lp-diagonal-exponent}
 (1-\eps_{q+1})\left(\frac13-\frac1{3-\eps_{q+1}}\right)
 =-\frac{\eps_{q+1}(1-\eps_{q+1})}{3(3-\eps_{q+1})}<0.
\end{equation}

\subsection{The new error estimate}\label{subsec:residual}
Apply \cref{lem:polarization} with $z=e^{2\pi iK_{q+1}x}$.  By \eqref{eq:b-properties},
\begin{equation}\label{eq:V-cubic}
 |V_{q+1}|^2V_{q+1}
 =-\sigma E_q
 +\sum_{\ell=1}^3H_{\ell,q+1}e^{2\pi i\ell K_{q+1}x},
\end{equation}
where
\begin{equation}\label{eq:H-defs}
 \begin{aligned}
 H_{1,q+1}&=a_{q+1}(a_{q+1}^2+2|b_{q+1}|^2),\\
 H_{2,q+1}&=b_{q+1}(2a_{q+1}^2+|b_{q+1}|^2),\\
 H_{3,q+1}&=a_{q+1}b_{q+1}^2.
 \end{aligned}
\end{equation}
Their Fourier supports satisfy
\begin{equation}\label{eq:H-support}
 \supp\widehat{H_{1,q+1}}, \supp\widehat{H_{3,q+1}}\subset[-2M_q,2M_q],
 \quad
 \supp\widehat{H_{2,q+1}}\subset[-3M_q,3M_q].
\end{equation}
By $\eqref{eq:h-properties}$, $h_{q+1}=1$ on the temporal support of $E_q$, and hence
\begin{equation}\label{eq:hE}
 h_{q+1}^3E_q=E_q.
\end{equation}
Therefore
\begin{equation}\label{eq:osc-exact}
 E_q+\sigma|w_{q+1}|^2w_{q+1}
 =E_q(1-\Theta_{q+1}^3)
 +\sigma h_{q+1}^3\Theta_{q+1}^3
 \sum_{\ell=1}^3H_{\ell,q+1}e^{2\pi i\ell K_{q+1}x}.
\end{equation}

Using the cubic identity 
\begin{equation}\label{eq:mixed-cubic}
 \begin{aligned}
 &|u_q+w_{q+1}|^2(u_q+w_{q+1})
 -|u_q|^2u_q-|w_{q+1}|^2w_{q+1}
 \\={}&2|u_q|^2w_{q+1}+u_q^2\overline{w_{q+1}}
 +2u_q|w_{q+1}|^2+\overline{u_q}w_{q+1}^2,
 \end{aligned}
\end{equation}
we obtain
\begin{equation}\label{eq:relaxed-next}
 i\partial_t u_{q+1}+\partial_x^2u_{q+1}
 +\sigma|u_{q+1}|^2u_{q+1}=E_{q+1},
\end{equation}
where we define
\begin{equation}\label{eq:error-decomp}
 \begin{aligned}
 E_{O,q+1}&:=E_q+\sigma|w_{q+1}|^2w_{q+1},\\
 E_{N,q+1}&:=\sigma\left(
 2|u_q|^2w_{q+1}+u_q^2\overline{w_{q+1}}
 +2u_q|w_{q+1}|^2+\overline{u_q}w_{q+1}^2
 \right),\\
 E_{D,q+1}&:=\partial_x^2w_{q+1},\\
 E_{T,q+1}&:=i\partial_t w_{q+1},\\
 E_{q+1}&:=E_{O,q+1}+E_{N,q+1}+E_{D,q+1}+E_{T,q+1}.
 \end{aligned}
\end{equation}
Hence $(u_{q+1},E_{q+1})$ is the new solution to our relaxed equation. The new defect $E_{q+1}$ consists of several types of errors arising from the perturbation and the old solution. The term $E_{O,q+1}$ is called the \textit{oscillation error} and $E_{N,q+1}$ is known as the \textit{Nash error}. The linear errors include $E_{D,q+1}$, the \textit{dispersion error}, and $E_{T,q+1}$, the \textit{temporal error}. We first study the Fourier support of the new error $E_{q+1}$, which is given by the next lemma.

\begin{lemma}\label{lem:new-error-support}
For every $t$, the new solution $(u_{q+1},E_{q+1})$ has Fourier support satisfying 
\begin{equation}\label{eq:new-error-support}
 \supp\widehat{u_{q+1}}(t,\cdot)
 \cup
 \supp\widehat{E_{q+1}}(t,\cdot)
 \subset[-5K_{q+1},5K_{q+1}].
\end{equation}
Moreover, the temporal support of both functions lies in $I_{q+1}$.
\end{lemma}

\begin{proof}
The linear errors have the same spatial support as $w_{q+1}$ given by \cref{lem:packet-support}.  The low-mode part $E_q(1-\Theta_{q+1}^3)$ is supported in $[-M_q-6\nu_{q+1},M_q+6\nu_{q+1}]$.  By \eqref{eq:H-support}, the high oscillation terms in \eqref{eq:osc-exact} lie within a radius of $6\nu_{q+1}+3M_q$, centered at one of the carriers $K_{q+1},2K_{q+1},3K_{q+1}$.

For the Nash terms, the largest carrier mode occurs in $\overline{u_q}w_{q+1}^2$.  Its carrier is $4K_{q+1}$, its slab width is at most $4\nu_{q+1}$, and all old modulations contribute at most $3M_q$.  Thus its frequencies have magnitude below
\[
 4K_{q+1}+4\nu_{q+1}+3M_q<5K_{q+1}.
\]
All other mixed terms are smaller.  This proves the spatial support claim.  Every term in \eqref{eq:error-decomp} is supported in time where $u_q$, $E_q$, or $h_{q+1}$ is supported, hence in $I_{q+1}$.
\end{proof}

We continue to estimate the four classes of errors. Recall that we have already fixed the parameters $s > 3$ and $\beta < \frac{1}{3}$ in \eqref{eq:s-beta}.

\begin{lemma}[Dispersion and temporal errors]\label{lem:linear-errors}
One has
\begin{equation}\label{eq:dispersion-error}
 \|E_{D,q+1}\|_{\A^{-s}_{x,t}}
 \leq C_q\lambda_{q+1}^{-\frac23(1-\eps_{q+1})},
\end{equation}
and
\begin{equation}\label{eq:temporal-error}
 \|E_{T,q+1}\|_{\A^{-s}_{x,t}}
 \leq C_q\left(
 \lambda_{q+1}^{-\frac23(1-\eps_{q+1})}
 +\lambda_{q+1}^{\beta-\frac23(1-\eps_{q+1})}
 \right).
\end{equation}
\end{lemma}

\begin{proof}
By the definition of the spacetime Wiener norm,
\[
\begin{aligned}
 \|E_{D,q+1}\|_{\A^{-s}_{x,t}}
 &= \|\partial_x^2w_{q+1}\|_{\A^{-s}_{x,t}} \\
 &\lesssim
 \sum_{k\in\Z}
 \langle k\rangle^{2-s}
 \big\|\widehat{w}_{q+1}(\cdot,k)\big\|_{C_t^0}.
\end{aligned}
\]
Since
\[
 \big\|\widehat{w}_{q+1}(\cdot,k)\big\|_{C_t^0}
 \leq \|w_{q+1}\|_{C_t^0L_x^1}
\]
for every \(k\in\Z\), and
\[
 \sum_{k\in\Z}\langle k\rangle^{2-s}<\infty
\]
because \(s>3\), we obtain
\[
 \|E_{D,q+1}\|_{\A^{-s}_{x,t}}
 \lesssim_s \|w_{q+1}\|_{C_t^0L_x^1}
 \leq
 C_q\lambda_{q+1}^{-\frac23(1-\eps_{q+1})}
\]
by \eqref{eq:w-L1}. For the temporal error $E_{T,q+1}$, notice that only $h_{q+1}$ and $b_{q+1}$ depend on time, so
\[
 \partial_t w_{q+1}
 =h_{q+1}'\Theta_{q+1} V_{q+1}
 +h_{q+1}\Theta_{q+1}(\partial_t b_{q+1})e^{4\pi iK_{q+1}x}.
\]
The $L^\infty$ norms of $V_{q+1}$ and $\partial_t b_{q+1}$ depend only on the stage-$q$ quantities. Using \eqref{eq:derivative-L1-to-A} with $m=0$, \eqref{eq:h-properties}, and \eqref{eq:Theta-L1-L2},
\[
 \|E_{T,q+1}\|_{\A^{-s}_{x,t}}
 \lesssim_q(1+\lambda_{q+1}^{\beta})\|\Theta_{q+1}\|_{L^1}
 \leq C_q(1+\lambda_{q+1}^\beta)\lambda_{q+1}^{-\frac23(1-\eps_{q+1})}.
\]
\end{proof}

\begin{lemma}[Nash error]\label{lem:Nash-error}
The Nash error can be bounded by
\begin{equation}\label{eq:Nash-error}
 \|E_{N,q+1}\|_{\A^{-s}_{x,t}}
 \leq C_q\left(
 \lambda_{q+1}^{-\frac23(1-\eps_{q+1})}
 +\lambda_{q+1}^{-\frac13(1-\eps_{q+1})}
 \right).
\end{equation}
\end{lemma}

\begin{proof}
Since the function $u_q$ is fixed, its $C_t^0L_x^\infty$ norm is absorbed into the constant $C_q$.  By H\"older's inequality,
\[
\begin{aligned}
 \big\||u_q|^2w_{q+1}\big\|_{C_t^0L_x^1}
 &\leq
 \|u_q\|_{C_t^0L_x^\infty}^2
 \|w_{q+1}\|_{C_t^0L_x^1},\\
 \big\|u_q^2\overline{w_{q+1}}\big\|_{C_t^0L_x^1}
 &\leq
 \|u_q\|_{C_t^0L_x^\infty}^2
 \|w_{q+1}\|_{C_t^0L_x^1}.
\end{aligned}
\]
Similarly,
\[
\begin{aligned}
 \big\|u_q|w_{q+1}|^2\big\|_{C_t^0L_x^1}
 &\leq
 \|u_q\|_{C_t^0L_x^\infty}
 \|w_{q+1}\|_{C_t^0L_x^2}^2,\\
 \big\|\overline{u_q}w_{q+1}^2\big\|_{C_t^0L_x^1}
 &\leq
 \|u_q\|_{C_t^0L_x^\infty}
 \|w_{q+1}\|_{C_t^0L_x^2}^2.
\end{aligned}
\]
It follows from \eqref{eq:error-decomp} and \eqref{eq:derivative-L1-to-A} with $m=0$ that
\[
 \|E_{N,q+1}\|_{\A^{-s}_{x,t}}
 \leq C_q\left(
 \|w_{q+1}\|_{C_t^0L_x^1}
 +\|w_{q+1}\|_{C_t^0L_x^2}^2
 \right).
\]
Utilizing \eqref{eq:w-L1} and \eqref{eq:w-L2}, we obtain
\[
 \|E_{N,q+1}\|_{\A^{-s}_{x,t}}
 \leq C_q\left(
 \lambda_{q+1}^{-\frac23(1-\eps_{q+1})}
 +\lambda_{q+1}^{-\frac13(1-\eps_{q+1})}
 \right),
\]
which proves \eqref{eq:Nash-error}.
\end{proof}

\begin{lemma}[Oscillation error]\label{lem:oscillation-error}
The low-output part satisfies
\begin{equation}\label{eq:principal-osc}
 \|E_q(1-\Theta_{q+1}^3)\|_{\A^{-s}_{x,t}}
 \leq C_q\delta_q\mu_{q+1}^{-s}.
\end{equation}
The high oscillation part satisfies
\begin{equation}\label{eq:carrier-osc}
 \left\|
 h_{q+1}^3\Theta_{q+1}^3H_{\ell,q+1}e^{2\pi i\ell K_{q+1}x}
 \right\|_{\A^{-s}_{x,t}}
 \leq C_q K_{q+1}^{-s}\lambda_{q+1}^{1-\eps_{q+1}},
 \qquad \ell=1,2,3.
\end{equation}
\end{lemma}

\begin{proof}
The zero Fourier coefficient of $\Theta_{q+1}^3$ is one by \eqref{eq:Theta-cubic}, so
\[
 1-\Theta_{q+1}^3=-\Pnz(\Theta_{q+1}^3).
\]
The product estimate \eqref{eq:Wiener-product}, the positive-Wiener bound \eqref{eq:old-error-positive-Wiener}, and \eqref{eq:Theta-Pneq} with $j=3$ give
\[
\begin{aligned}
 \|E_q(1-\Theta_{q+1}^3)\|_{\A^{-s}_{x,t}}
 &\lesssim_s
 \|E_q\|_{\A^s_{x,t}}
 \|\Pnz(\Theta_{q+1}^3)\|_{\A^{-s}_x}\\
 &\leq C_q\delta_q\mu_{q+1}^{-s}.
\end{aligned}
\]
For the high oscillatory terms, \eqref{eq:H-support} and \eqref{eq:Theta-Fourier-support} imply
\[
 \supp\widehat{\Theta_{q+1}^3H_{\ell,q+1}}
 \subset[-6\nu_{q+1}-3M_q,6\nu_{q+1}+3M_q]
 \subset[-K_{q+1}/2,K_{q+1}/2].
\]
Apply \eqref{eq:modulation-Wiener} and then \eqref{eq:Theta-A0}:
\[
 \left\|h_{q+1}^3\Theta_{q+1}^3H_{\ell,q+1}e^{2\pi i\ell K_{q+1}x}\right\|_{\A^{-s}_{x,t}}
 \leq C_q K_{q+1}^{-s}\|\Theta_{q+1}^3\|_{\A_x^0}
 \leq C_q K_{q+1}^{-s}\lambda_{q+1}^{1-\eps_{q+1}}.
\]
\end{proof}

Combining the preceding estimates gives the overall error estimate for $E_{q+1}$.

\begin{proposition}\label{prop:one-step-residual}
Under the assumptions in \cref{prop:main-inductive},
\begin{equation}\label{eq:one-step-residual}
 \|E_{q+1}\|_{\A^{-s}_{x,t}}
 \leq C_q\left[
 \delta_q\mu_{q+1}^{-s}
 +K_{q+1}^{-s}\lambda_{q+1}^{1-\eps_{q+1}}
 +\lambda_{q+1}^{-\frac13(1-\eps_{q+1})}
 +\lambda_{q+1}^{\beta-\frac23(1-\eps_{q+1})}
 \right].
\end{equation}
Each power of $\lambda_{q+1}$ on the right is negative.
\end{proposition}

\begin{proof}
The estimate follows from \cref{lem:linear-errors,lem:Nash-error,lem:oscillation-error}; the omitted powers $\lambda_{q+1}^{-\frac23(1-\eps_{q+1})}$ are dominated by $\lambda_{q+1}^{-\frac13(1-\eps_{q+1})}$.  Since $\mu_{q+1}=\lambda_{q+1}^{\eps_{q+1}}$ and $K_{q+1}=16\lambda_{q+1}^{1+\eps_{q+1}}$, the first two powers are
\[
 -s\eps_{q+1}
 \qquad\text{and}\qquad
 1-\eps_{q+1}-(1+\eps_{q+1})s.
\]
They are negative, as the fixed factor $\delta_q$ in the first term does not affect its $\lambda_{q+1}$ exponent.  The Nash exponent is negative, and
\[
 \beta-\frac23(1-\eps_{q+1})
 \leq\beta-\frac13<0
\]
because $\eps_{q+1}\leq1/2$ and $\beta<1/3$.
\end{proof}

\subsection{Fourier summability}\label{subsec:absolute}

This subsection proves the estimate \eqref{eq:CS-induction-main} that defines the limiting cubic product. For $j=1,2$, define the sequence of old Fourier coefficients
\begin{equation}\label{eq:cj-def}
 c_{1,q+1}(t,\xi)=a_{q+1}\one_{\{\xi=0\}},
 \qquad
 c_{2,q+1}(t,\xi)=\widehat{b_{q+1}}(t,\xi).
\end{equation}
Then
\begin{equation}\label{eq:packet-Fourier-formula}
 \widehat{w_{q+1}^{(j)}}(t,n)
 =h_{q+1}(t)
 \sum_{\eta\in\mu_{q+1}\Z,\ |\eta|\leq2\nu_{q+1}}
 \widehat{\Theta_{q+1}}(\eta)c_{j,q+1}(t,n-jK_{q+1}-\eta).
\end{equation}
The sequence $c_{2,q+1}$ is supported in $[-M_q,M_q]$, whereas $c_{1,q+1}$ is supported only at $\{0\}$.  Since $100 M_q<\mu_{q+1}$, the intervals
\[
 \eta+[-M_q,M_q],
 \qquad \eta\in\mu_{q+1}\Z,
\]
are pairwise disjoint.  Hence, for every $j\in\{1,2\}$ and $n\in\Z$, at most one summand on the right-hand side of \eqref{eq:packet-Fourier-formula} is nonzero. Since the old stage is fixed, we have
\begin{equation}\label{eq:cj-A0-fixed}
 \sum_\xi\|c_{j,q+1}(\cdot,\xi)\|_{C_t^0}
 +\sum_k\|\widehat{u_q}(\cdot,k)\|_{C_t^0}
 \leq C_q.
\end{equation}

The expansion of the cubic functional $\CS_s(u_{q+1}) = \CS_s(u_q+w_{q+1})$ contains exactly eight placements:
\begin{equation}\label{eq:CS-eight}
 \begin{aligned}
    \CS_s(u_q+w_{q+1})
 \leq{}&\CS_s(u_q,u_q,u_q)\\
 &+\CS_s(w_{q+1},u_q,u_q)+\CS_s(u_q,w_{q+1},u_q)+\CS_s(u_q,u_q,w_{q+1})\\
 &+\CS_s(w_{q+1},w_{q+1},u_q)+\CS_s(w_{q+1},u_q,w_{q+1})+\CS_s(u_q,w_{q+1},w_{q+1})\\
 &+\CS_s(w_{q+1},w_{q+1},w_{q+1})\\
 &:= \CS_s(u_q,u_q,u_q) + I_1 + I_2 + I_3,
 \end{aligned}
\end{equation}
where 
\begin{align}
I_1&=\CS_s(w_{q+1},u_q,u_q)+\CS_s(u_q,w_{q+1},u_q)+
\CS_s(u_q,u_q,w_{q+1}), \label{def:I_1} \\
I_2 &= \CS_s(w_{q+1},w_{q+1},u_q)+\CS_s(w_{q+1},u_q,w_{q+1})+\CS_s(u_q,w_{q+1},w_{q+1}), \label{def:I_2} \\
I_3 &= \CS_s(w_{q+1},w_{q+1},w_{q+1}). \label{def:I_3}
\end{align}
We will continue to estimate $I_1$, $I_2$ and $I_3$ separately. Recall from \eqref{def:absolute-cubic} that, for example,
\[
\CS_s(\omega_{q+1},u_q,u_q) = \sum_{n_1,n_2,n_3\in\Z}
 \wt{n_1-n_2+n_3}^{-s}
 \left\|
 \widehat \omega_{q+1}(\cdot,n_1)
 \overline{\widehat u_q(\cdot,n_2)}
 \widehat u_q(\cdot,n_3)
 \right\|_{C_t^0}.
\]
Hence, to estimate $I_k,\  k = 1,2,3$, it will be helpful to first identify the Fourier support of the cubic interaction. This is discussed in the next lemma.

\subsubsection{Fourier support of cubic functionals}
We label the three slots of a cubic interaction
$f_1\overline{f_2}f_3$ by $1,2,3$, respectively.
Whenever slot $r$ contains the perturbation $\omega_{q+1}$, we use the expansion $w_{q+1}=w_{q+1}^{(1)}+w_{q+1}^{(2)}$ from \eqref{eq:packets} and denote the
selected packet by $w_{q+1}^{(j_r)}$, where $j_r\in\{1,2\}$.
Thus $j_r$ specifies the packet occupying slot $r$,
whose carrier frequency is $j_rK_{q+1}$.
The second slot is conjugated, so its carrier enters
the output frequency with a minus sign.

\begin{lemma}\label{lem:carrier-separation}
Suppose $K_{q+1}=16\nu_{q+1}$ and $\mu_{q+1}>100M_q$.
Let $I_1,I_2,I_3$ be given by \eqref{def:I_1}-\eqref{def:I_3}.
Then the following statements hold.
\begin{enumerate}[label=\textup{(\roman*)}]
 \item Every interaction in $I_1$ has output frequency
 of magnitude at least $K_{q+1}/2$.

 \item In $I_2$, the carrier vanishes precisely when the new
 factors occupy slots $(1,2)$ with $j_1=j_2$, or slots $(2,3)$
 with $j_2=j_3$. The remaining slab frequency is
 $\eta_1-\eta_2$ or $-\eta_2+\eta_3$, respectively.
 Every other interaction has output frequency of magnitude
 at least $K_{q+1}/2$.

 \item In $I_3$, the carrier vanishes only for
 $(j_1,j_2,j_3)=(1,2,1)$.
 For this interaction, the output frequency is
 $\eta_1-\eta_2+\eta_3-\xi_2$, where $\xi_2$ is a Fourier
 frequency of $b_{q+1}$.
 Every other packet triple has output frequency of magnitude
 at least $K_{q+1}/2$.
\end{enumerate}
\end{lemma}

\begin{proof}
A Fourier frequency of the perturbation $w_{q+1}^{(j_r)}$ has the form
\[
 n_r=j_rK_{q+1}+\eta_r+\xi_r,
 \qquad
 \eta_r\in\mu_{q+1}\Z,
 \qquad
 |\eta_r|\leq2\nu_{q+1}.
\]
Here $\xi_r=0$ if $j_r=1$, since the first packet has the
spatially constant amplitude $a_{q+1}$.
If $j_r=2$, then $\xi_r$ is a Fourier frequency of $b_{q+1}$
and satisfies $|\xi_r|\leq M_q$.
We denote the frequencies of the old factor $u_q$ by $k_r$, and so $|k_r|\leq M_q$.

The output frequency of the cubic interaction is $n_1-n_2+n_3$.
For an interaction containing $m$ perturbations, it can be written as
\[
 n=cK_{q+1}+R,
 \qquad
 |R|\leq2m\nu_{q+1}+3M_q<K_{q+1}/2,
 \qquad 1\leq m\leq3,
\]
where $c$ is the sum of the packet labels.
The last inequality follows from $\mu_{q+1}\leq\nu_{q+1}$,
$\mu_{q+1}>100M_q$, and $K_{q+1}=16\nu_{q+1}$.
Since $c$ is an integer, every nonzero carrier gives
\[
 |n|\geq |c|K_{q+1}-|R|\geq K_{q+1}/2.
\]

For terms appearing in $I_1$, $c=j_r$ for $r=1,3$ and $c=-j_2$
for $r=2$. In every case $c\neq0$, proving \textup{(i)}. In the case of $I_2$, the three placements have carrier coefficients
\begin{equation}\label{eq:two-carrier-placement-table}
 \begin{array}{c|c}
 \text{perturbation slots}&\text{carrier coefficient}\\
 \hline
 (1,2)&c = j_1-j_2\\
 (1,3)&c= j_1+j_3\\
 (2,3)&c= j_3-j_2.
 \end{array}
\end{equation}
For example, the first row corresponds to
$w_{q+1}^{(j_1)}\overline{w_{q+1}^{(j_2)}}u_q$,
whose output frequency is
\[
 n=(j_1-j_2)K_{q+1}
   +(\eta_1-\eta_2)+(\xi_1-\xi_2)+k_3.
\]
The carrier vanishes exactly when $j_1=j_2$.
The third row is analogous, with remaining slab frequency
$-\eta_2+\eta_3$.
In the middle row, $j_1+j_3\geq2$, so the carrier cannot vanish.
This proves \textup{(ii)}.

The interaction in $I_3$ is given by
$w_{q+1}^{(j_1)}\overline{w_{q+1}^{(j_2)}}w_{q+1}^{(j_3)}$,
and its output frequency is
\[
 n=(j_1-j_2+j_3)K_{q+1}
   +(\eta_1-\eta_2+\eta_3)
   +(\xi_1-\xi_2+\xi_3).
\]
The complete carrier table is
\begin{equation}\label{eq:three-carrier-table}
 \begin{array}{c|cccccccc}
 (j_1,j_2,j_3)
 &(1,1,1)&(1,1,2)&(1,2,1)&(1,2,2)
 &(2,1,1)&(2,1,2)&(2,2,1)&(2,2,2)\\
 \hline
 j_1-j_2+j_3&1&2&0&1&2&3&1&2.
 \end{array}
\end{equation}
The only zero carrier occurs for $(j_1,j_2,j_3)=(1,2,1)$.
In this case $\xi_1=\xi_3=0$, so
\[
 n=\eta_1-\eta_2+\eta_3-\xi_2.
\]
This proves \textup{(iii)}.
\end{proof}
We now develop the bounds for $I_1$, $I_2$ and $I_3$. 
\subsubsection{Estimates for the cubic functionals}\label{subsubsec:one-new}

By \cref{lem:carrier-separation}, all three terms in $I_1$ involve output frequencies of magnitude at least $K_{q+1}/2$. Summing the absolute coefficients using \eqref{eq:cj-A0-fixed} gives
\begin{equation}\label{eq:one-new-estimate}
 \begin{aligned}
 &\CS_s(w_{q+1},u_q,u_q)+\CS_s(u_q,w_{q+1},u_q)+\CS_s(u_q,u_q,w_{q+1})\\
 &\qquad\leq C_q K_{q+1}^{-s}\|\Theta_{q+1}\|_{\A_x^0}
 \leq C_q K_{q+1}^{-s}\lambda_{q+1}^{\frac13(1-\eps_{q+1})}.
 \end{aligned}
\end{equation}
We summarize the estimate of $I_2$ in the following proposition.
\begin{proposition}\label{prop:two-new}
It holds that
\begin{equation}\label{eq:two-new-estimate}
 \begin{aligned}
 I_2 \leq
 C_q\left[
 K_{q+1}^{-s}\lambda_{q+1}^{\frac23(1-\eps_{q+1})}
 +\lambda_{q+1}^{-\frac13(1-\eps_{q+1})}
 +\lambda_{q+1}^{-\frac13(1-\eps_{q+1})}\mu_{q+1}^{-s}
 \right].
 \end{aligned}
\end{equation}
The low-mode terms are the same-packet pairs in the placements $(1,2)$ and $(2,3)$ of \eqref{eq:two-carrier-placement-table}.
\end{proposition}

\begin{proof}
We begin by writing all signed outputs.

\smallskip
\noindent\emph{Placement $\CS_s(w_{q+1},w_{q+1},u_q)$.}
Choose packet labels $j_1,j_2\in\{1,2\}$.  The signed output is
\begin{equation}\label{eq:two-new-12-output}
 (j_1-j_2)K_{q+1}
 +(\eta_1-\eta_2)
 +(\xi_1-\xi_2)
 +k_3.
\end{equation}
If $j_1\neq j_2$, the carrier is $\pm K_{q+1}$, so the term has high output frequency.  If $j_1=j_2=j$, the carrier vanishes and the slab output is
\[
 r:=\eta_1-\eta_2\in\mu_{q+1}\Z.
\]

\smallskip
\noindent\emph{Placement $\CS_s(w_{q+1},u_q,w_{q+1})$.}
Choose packet labels $j_1,j_3\in\{1,2\}$.  The signed output is
\begin{equation}\label{eq:two-new-13-output}
 (j_1+j_3)K_{q+1}
 +(\eta_1+\eta_3)
 +(\xi_1+\xi_3)
 -k_2.
\end{equation}
The carrier coefficient belongs to $\{2,3,4\}$, so every label choice is high.  No low interaction is possible when both new factors occupy the two nonconjugated slots.

\smallskip
\noindent\emph{Placement $\CS_s(u_q,w_{q+1},w_{q+1})$.}
Choose packet labels $j_2,j_3\in\{1,2\}$.  The signed output is
\begin{equation}\label{eq:two-new-23-output}
 (-j_2+j_3)K_{q+1}
 +(-\eta_2+\eta_3)
 +(-\xi_2+\xi_3)
 +k_1.
\end{equation}
If $j_2\neq j_3$, the carrier is $\pm K_{q+1}$.  If $j_2=j_3=j$, it vanishes and the slab output is
\[
 r:=-\eta_2+\eta_3\in\mu_{q+1}\Z.
\]

All high cases in \eqref{eq:two-new-12-output}--\eqref{eq:two-new-23-output} have output magnitude at least $K_{q+1}/2$ by \cref{lem:carrier-separation}.  There are eight such label-placement combinations: two unequal-label cases in each of the placements $(1,2)$ and $(2,3)$, and four cases in the placement $(1,3)$.  Summing their Fourier coefficients gives
\begin{equation}\label{eq:two-new-high}
 \text{all high-mode terms}
 \leq C_q K_{q+1}^{-s}\|\Theta_{q+1}\|_{\A_x^0}^2
 \leq C_q K_{q+1}^{-s}\lambda_{q+1}^{\frac23(1-\eps_{q+1})}.
\end{equation}

It remains to estimate the four low cases: $j=1,2$ in \eqref{eq:two-new-12-output} and $j=1,2$ in \eqref{eq:two-new-23-output}. For the placement $(1,2)$ and a fixed packet label $j$, define the nonnegative coefficient sum
\begin{equation}\label{eq:G12}
 G_{j,q+1}^{12}(\kappa)
 :=\sum_{\xi_1-\xi_2+k_3=\kappa}
 \left\|
 c_{j,q+1}(\cdot,\xi_1)
 \overline{c_{j,q+1}(\cdot,\xi_2)}
 \widehat{u_q}(\cdot,k_3)
 \right\|_{C_t^0}.
\end{equation}
It is supported in $|\kappa|\leq3M_q$, and
\begin{equation}\label{eq:G12-l1}
 \sum_\kappa G_{j,q+1}^{12}(\kappa)\leq C_q
\end{equation}
by \eqref{eq:cj-A0-fixed}.  The corresponding contribution is bounded by
\begin{equation}\label{eq:low12-sum}
 \CS_s(w_{q+1}^{(j)},w_{q+1}^{(j)},u_q) \leq \sum_{r\in\mu_{q+1}\Z}
 \sum_{\kappa\in\Z}
 \wt{r+\kappa}^{-s}
 G_{j,q+1}^{12}(\kappa)
 Q_{2,q+1}(r),
\end{equation}
where
\begin{equation}\label{eq:Q2-def}
 Q_{2,q+1}(r)
 :=\sum_{\eta_1-\eta_2=r}
 \widehat{\Theta_{q+1}}(\eta_1)\widehat{\Theta_{q+1}}(\eta_2)
 =\widehat{\Theta_{q+1}^2}(r)\geq0.
\end{equation}
Here the identity with $\widehat{\Theta_{q+1}^2}$ uses the evenness of $\widehat{\Theta_{q+1}}$; see \cref{prop:Theta}. For $r=0$, \eqref{eq:Theta-Q2} and \eqref{eq:G12-l1} give
\begin{equation}\label{eq:low12-r0}
 \sum_\kappa\wt\kappa^{-s}G_{j,q+1}^{12}(\kappa)Q_{2,q+1}(0)
 \leq C_q\int_{\T}\Theta_{q+1}^2\,\dd x
 \leq C_q\lambda_{q+1}^{-\frac13(1-\eps_{q+1})}.
\end{equation}
If $r\neq0$, then $|r|\geq\mu_{q+1}$ and $|\kappa|\leq3M_q<3\mu_{q+1}/100$, so
\[
 \wt{r+\kappa}^{-s}\lesssim_s\wt r^{-s}.
\]
Consequently,
\begin{equation}\label{eq:low12-rneq0}
 \begin{aligned}
 &\sum_{r\neq0}\sum_\kappa
 \wt{r+\kappa}^{-s}G_{j,q+1}^{12}(\kappa)Q_{2,q+1}(r)\\
 &\qquad\leq C_q
 \sum_{r\neq0}\wt r^{-s}Q_{2,q+1}(r)
 =C_q\|\Pnz(\Theta_{q+1}^2)\|_{\A_x^{-s}}\\
 &\qquad\leq C_q
 \lambda_{q+1}^{-\frac13(1-\eps_{q+1})}\mu_{q+1}^{-s},
 \end{aligned}
\end{equation}
where the last line follows from \eqref{eq:Theta-Pneq} with $j=2$. For the placement $(2,3)$, define
\begin{equation}\label{eq:G23}
 G_{j,q+1}^{23}(\kappa)
 :=\sum_{-\xi_2+\xi_3+k_1=\kappa}
 \left\|
 \widehat{u_q}(\cdot,k_1)
 \overline{c_{j,q+1}(\cdot,\xi_2)}
 c_{j,q+1}(\cdot,\xi_3)
 \right\|_{C_t^0}.
\end{equation}
Again $G_{j,q+1}^{23}$ is supported in $|\kappa|\leq3M_q$ and has $\ell^1$ size at most $C_q$.  Replacing $r=\eta_1-\eta_2$ by $r=-\eta_2+\eta_3$ gives exactly the same $Q_{2,q+1}(r)$, because again the Fourier coefficients of $\Theta_{q+1}$ are real and even. The estimates \eqref{eq:low12-r0} and \eqref{eq:low12-rneq0} therefore apply verbatim, and hence 
\begin{equation}\label{eq:C23}
\CS_s(u_q,w_{q+1}^{(j)},w_{q+1}^{(j)}) \leq C_q \lambda_{q+1}^{-\frac13(1-\eps_{q+1})} + C_q \lambda_{q+1}^{-\frac13(1-\eps_{q+1})}\mu_{q+1}^{-s}. 
\end{equation}

Finally, summing \eqref{eq:two-new-high}, \eqref{eq:low12-r0}, \eqref{eq:low12-rneq0} and \eqref{eq:C23} proves \eqref{eq:two-new-estimate}.
\end{proof}

Now we look at the terms in $I_3$. The packet carrier table \eqref{eq:three-carrier-table} shows that the unique low carrier is $(1,2,1)$.  We first compute it exactly.

\begin{lemma}\label{lem:exceptional-triad}
We have that
\begin{equation}\label{eq:exceptional-bound}
 \CS_s(w_{q+1}^{(1)},w_{q+1}^{(2)},w_{q+1}^{(1)})
 \leq
 \|E_q\|_{\A^{-s}_{x,t}}
 +C_q\mu_{q+1}^{-s}.
\end{equation}
\end{lemma}

\begin{proof}
Using \eqref{eq:packets}, the cubic interaction is
\[
 w_{q+1}^{(1)}\overline{w_{q+1}^{(2)}}w_{q+1}^{(1)}
 =h_{q+1}^3a_{q+1}^2\overline{b_{q+1}}\,\Theta_{q+1}^3
 =-\sigma E_q\Theta_{q+1}^3,
\]
where we used $h_{q+1}^3E_q=E_q$ and $a_{q+1}^2\overline{b_{q+1}}=-\sigma E_q$. Set 
\[
J := \CS_s(w_{q+1}^{(1)},w_{q+1}^{(2)},w_{q+1}^{(1)}).
\]
Then \eqref{eq:packet-Fourier-formula} gives
\begin{equation}\label{eq:exceptional-Fourier-sum}
 J = a_{q+1}^2\sum_{\eta_1,\eta_2,\eta_3,\xi}
 \wt{\eta_1-\eta_2+\eta_3-\xi}^{-s}
 \|h_{q+1}^3\overline{\widehat{b_{q+1}}(\cdot,\xi)}\|_{C_t^0}
 \prod_{r=1}^3\widehat{\Theta_{q+1}}(\eta_r).
\end{equation}
From $b_{q+1}=-\sigma a_{q+1}^{-2}\overline{E_q}$,
\begin{equation}\label{eq:b-Fourier-E}
 \overline{\widehat{b_{q+1}}(t,\xi)}
 =-\sigma a_{q+1}^{-2}\widehat{E_q}(t,-\xi).
\end{equation}
Write $k=-\xi$ and
\[
 r=\eta_1-\eta_2+\eta_3\in\mu_{q+1}\Z.
\]
When $r=0$, \eqref{eq:Theta-Q3} implies that the right-hand side of \eqref{eq:exceptional-Fourier-sum} is equal to
\[
 \sum_k\wt k^{-s}\|\widehat{E_q}(\cdot,k)\|_{C_t^0}
 \sum_{\eta_1-\eta_2+\eta_3=0}
 \prod_{r=1}^3\widehat{\Theta_{q+1}}(\eta_r)
 =\|E_q\|_{\A^{-s}_{x,t}}.
\]
If $r\neq0$, then $|r|\geq\mu_{q+1}$ and $|k|\leq M_q<\mu_{q+1}/100$, so
\[
 \wt{r+k}^{-s}\lesssim_s\wt r^{-s}.
\]
Since $\widehat{\Theta_{q+1}}$ is even,
\[
 \sum_{\eta_1-\eta_2+\eta_3=r}
 \prod_{j=1}^3\widehat{\Theta_{q+1}}(\eta_j)
 =\widehat{\Theta_{q+1}^3}(r).
\]
Summing first in $k = -\xi$ and then in the slab variables in \eqref{eq:exceptional-Fourier-sum} therefore yields
\begin{align*}
     J &\leq \|E_q\|_{\A^{-s}_{x,t}} + C_q\sum_{r\neq0}\wt r^{-s}\widehat{\Theta_{q+1}^3}(r)
 =\|E_q\|_{\A^{-s}_{x,t}}+ C_q\|\Pnz(\Theta_{q+1}^3)\|_{\A_x^{-s}}\\ 
 &\leq \|E_q\|_{\A^{-s}_{x,t}}+ C_q\mu_{q+1}^{-s} 
\end{align*}
by \eqref{eq:Theta-Pneq}.  This proves \eqref{eq:exceptional-bound}.
\end{proof}

Every other packet triple has a nonzero carrier.  Its Fourier support lies at frequencies of magnitude at least $K_{q+1}/2$, and its coefficient sum is bounded by the product of three slab $\A^0$ masses and fixed old-amplitude bounds. Therefore
\begin{equation}\label{eq:pure-high}
 \sum_{(j_1,j_2,j_3)\neq(1,2,1)}
 \CS_s(w_{q+1}^{(j_1)},w_{q+1}^{(j_2)},w_{q+1}^{(j_3)})
 \leq C_q K_{q+1}^{-s}\|\Theta_{q+1}\|_{\A_x^0}^3
 \leq C_q K_{q+1}^{-s}\lambda_{q+1}^{1-\eps_{q+1}},
\end{equation}
where we used \eqref{eq:modulation-Wiener}. Combining \eqref{eq:exceptional-bound} and \eqref{eq:pure-high} gives
\begin{equation}\label{eq:pure-new-total}
 I_3 \leq
 \|E_q\|_{\A^{-s}_{x,t}}
 +C_q\left(\mu_{q+1}^{-s}+K_{q+1}^{-s}\lambda_{q+1}^{1-\eps_{q+1}}\right).
\end{equation}

\begin{proposition}\label{prop:absolute-increment}
Under the assumptions in \cref{lem:carrier-separation}, we have
\begin{equation}\label{eq:absolute-increment}
 \CS_s(u_{q+1})
 \leq
 \CS_s(u_q)
 +\|E_q\|_{\A^{-s}_{x,t}}
 +C_q\left[
 \mu_{q+1}^{-s}
 +\lambda_{q+1}^{-\frac13(1-\eps_{q+1})}
 +K_{q+1}^{-s}\lambda_{q+1}^{1-\eps_{q+1}}
 \right].
\end{equation}
\end{proposition}

\begin{proof}
Insert \eqref{eq:one-new-estimate}, \eqref{eq:two-new-estimate}, and \eqref{eq:pure-new-total} into the expansion \eqref{eq:CS-eight}.  The terms
\[
 K_{q+1}^{-s}\lambda_{q+1}^{\frac13(1-\eps_{q+1})},
 \qquad
 K_{q+1}^{-s}\lambda_{q+1}^{\frac23(1-\eps_{q+1})}
\]
are dominated by $K_{q+1}^{-s}\lambda_{q+1}^{1-\eps_{q+1}}$ for $\lambda_{q+1}\geq1$.  The term
$\lambda_{q+1}^{-\frac13(1-\eps_{q+1})}\mu_{q+1}^{-s}$ is dominated by $\lambda_{q+1}^{-\frac13(1-\eps_{q+1})}$.  This gives \eqref{eq:absolute-increment}.
\end{proof}

\subsection{Parameter choice and completion of the induction}\label{subsec:parameter-choice}

We now finish the proof of \cref{prop:main-inductive}.

\begin{proof}[Proof of \cref{prop:main-inductive}]
We start by showing the base case. Choose a nonzero real function $\chi_0\in C_c^\infty((7/16,9/16))$ and a small scalar $\eta>0$. Let
\begin{equation}\label{eq:seed-u0}
 u_0(t,x)=\eta\chi_0(t)e^{2\pi ix}
\end{equation}
and
\begin{equation}\label{eq:seed-E0}
 E_0=i\partial_tu_0+\partial_x^2u_0+\sigma|u_0|^2u_0.
\end{equation}
Explicitly,
\[
 E_0
 =\eta(i\chi_0'-4\pi^2\chi_0)e^{2\pi ix}
 +\sigma\eta^3|\chi_0|^2\chi_0e^{2\pi ix}.
\]
Both functions are compactly supported smooth trigonometric polynomials with spatial frequency $1$.  Choose $\eta$ small enough that
\[
 \|E_0\|_{\A^{-s}_{x,t}}\leq\delta_0.
\]
Set $M_0=10$ and, consistently with \eqref{eq:Iq-explicit}, set
\begin{equation}\label{eq:I0-explicit}
 I_0=\left[\frac7{16},\frac9{16}\right].
\end{equation}
The identity \eqref{eq:seed-E0} is the relaxed equation
\eqref{eq:relaxed-main-prop} at stage $q=0$. The interval
\eqref{eq:I0-explicit} is exactly \eqref{eq:Iq-explicit} for $q=0$, and
$\chi_0\in C_c^\infty((7/16,9/16))$ ensures that both temporal supports
are contained in $I_0$.  Since $u_0$ and $E_0$ have only spatial
frequency $1$ and $M_0=10$, the support condition
\eqref{eq:uq-Eq-support} also holds.  The choice of $\eta$ gives
\eqref{eq:error-budget-main}.  Finally,
$m_0=\|u_0\|_{C_t^0L_x^1}>0$, the base case of
\eqref{eq:CS-induction-main} is the identity
$\CS_s(u_0)\leq\CS_s(u_0)$, and \eqref{eq:nontriviality-main} follows
from $\|u_0-u_0\|_{C_t^0L_x^1}=0$.

Assume the sequence has been constructed through stage $q$.  Choose $A_{q+1}$ by \eqref{eq:A-choice}, then define $a_{q+1}$ and $b_{q+1}$ by \eqref{eq:a-b-choice}.  At this point every old-stage norm and every amplitude norm is fixed.  Next fix the prescribed exponent $\eps_{q+1} = 2^{-(q+1)}$. Choose
\[
 \lambda_{q+1}=2^{L_{q+1}}
\]
with $L_{q+1}$ divisible by $2^{q+1}$.  This makes $\lambda_{q+1}^{\eps_{q+1}}$ and $\lambda_{q+1}^{1-\eps_{q+1}}$ integers.  Enlarge $L_{q+1}$ successively until all of the following conditions hold:
\begin{enumerate}[label=\textup{(\arabic*)}]
 \item the slab normalization in \cref{prop:Theta} is valid;
 \item $\mu_{q+1} = \lambda_{q+1}^{\eps_{q+1}}>100M_q$; \label{parameter_req_2}
 \item $\lambda_{q+1}^{-\beta}<2^{-q-10}$ and
 \[
  \sum_{j=1}^{q+1}\lambda_j^{-\beta}<\frac1{16}; \label{parameter_req_3}
 \]
 \item the right side of \eqref{eq:one-step-residual} is below $\delta_{q+1}= 2^{-q-201}$; \label{parameter_req_4}
 \item the full remainder term $C_q[\cdots]$ in \eqref{eq:absolute-increment} is below $\gamma_{q+1}$;\label{parameter_req_5}
 \item the estimates \eqref{eq:w-Halpha} and \eqref{eq:w-Lp}, with $\alpha=\alpha_{q+1}$ and $p=p_{q+1}$, are each below $2^{-q-11}$;\label{parameter_req_6}
 \item the $L^1$ estimate \eqref{eq:w-L1} is below $2^{-q-3}m_0$. \label{parameter_req_7}
\end{enumerate}
These requirements are compatible.  Indeed, the exponents on the right-hand side of \eqref{eq:one-step-residual} are negative by \cref{prop:one-step-residual}. The remainder terms in $\eqref{eq:absolute-increment}$ contain the negative powers
\[
 -s\eps_{q+1},
 \qquad
 -\frac13(1-\eps_{q+1}),
 \qquad
 1-\eps_{q+1}-(1+\eps_{q+1})s.
\]
The Sobolev and $L^{p_{q+1}}$ exponents are negative by \eqref{eq:H-diagonal-exponent} and \eqref{eq:Lp-diagonal-exponent}.  Finally, $\lambda_{q+1}^{-\beta}$ and the $L^1$ increment tend to zero as $\lambda_{q+1}\to\infty$.

Now, construct $h_{q+1}$, $\Theta_{q+1}$, and $w_{q+1}$ as in
\cref{subsec:one-step}. The solution to the equation
\eqref{eq:relaxed-main-prop} at stage $q+1$ is then given by
\eqref{eq:packets}, \eqref{eq:relaxed-next} and \eqref{eq:error-decomp}. The identity \eqref{eq:Iqplus1-explicit-step} is \eqref{eq:Iq-explicit} with $q$
replaced by $q+1$.  Together with the cutoff support in
\eqref{eq:h-properties} and the temporal conclusion of
\cref{lem:new-error-support}, it gives
\[
 \supp_tu_{q+1}\cup\supp_tE_{q+1}\subset I_{q+1}.
\]
Moreover, since the first inequality in the parameter requirement \ref{parameter_req_3} is imposed at every stage,
\[
 \sum_{j=1}^{\infty}\lambda_j^{-\beta}
 <\sum_{j=1}^{\infty}2^{-j-9}
 =2^{-9}<\frac1{16},
\]
which yields \eqref{eq:time-sum}.

The divisibility by $2^{q+1}$ imposed on $L_{q+1}$ ensures that the parameters in
\eqref{eq:mu-nu-K} are all integers. Parameter requirement~\ref{parameter_req_2}
gives the inequality in \eqref{eq:old-new-separation-main}, while
$M_{q+1}=5K_{q+1}$ holds by definition in
\eqref{eq:one-step-frequency-definitions}.
Furthermore, \eqref{eq:packet-support-1}--\eqref{eq:packet-support-2}
give the exact block inclusion \eqref{eq:exact-block}. With this choice
of $M_{q+1}$, the bounds on the explicit packet sets in the proof of
\cref{lem:packet-support} give the annular localization
\eqref{eq:block-annulus}. The blocks are pairwise disjoint because
\eqref{eq:block-annulus} places the new block above $2M_q$, whereas the
inductively constructed earlier blocks lie below $M_q$. Finally,
\eqref{eq:new-error-support}, together with $M_{q+1}=5K_{q+1}$, proves \eqref{eq:uq-Eq-support} at stage $q+1$.

The parameter requirement \ref{parameter_req_6} together with \eqref{eq:w-Halpha} and
\eqref{eq:w-Lp} gives \eqref{eq:small-increments-main}, while the requirement \ref{parameter_req_7}
and \eqref{eq:w-L1} give \eqref{eq:L1-budget-main}.  The
requirement \ref{parameter_req_4} and \eqref{eq:one-step-residual} give
\eqref{eq:error-budget-main} at stage $q+1$.  By
\eqref{eq:absolute-increment}, the requirement \ref{parameter_req_5}, and the inductive
bound \eqref{eq:error-budget-main} at stage $q$,
\[
 \CS_s(u_{q+1})
 \leq\CS_s(u_q)+\delta_q+\gamma_{q+1}
 \leq
 \CS_s(u_0)+\sum_{j=0}^{q}(\delta_j+\gamma_{j+1}),
\]
which is \eqref{eq:CS-induction-main} at stage $q+1$.  Finally,
\[
 \|u_{q+1}-u_0\|_{C_t^0L_x^1}
 \leq\sum_{j=0}^q\|w_{j+1}\|_{C_t^0L_x^1}
 <\sum_{j=0}^q2^{-j-3}m_0
 <\frac14m_0<\frac12m_0,
\]
which is \eqref{eq:nontriviality-main}.  The induction is then complete.
\end{proof}

\section*{Acknowledgments}
The author is grateful to Xiaoan Shen and Christof Sparber for stimulating discussions on the cubic NLS equation. The author used OpenAI's ChatGPT as an interactive aid in checking portions of the mathematical argument and in editing the exposition. The author takes full responsibility for the content.


\begin{thebibliography}{99}

\bibitem{ABGN}
E.~Ashkarian, A.~Bhargava, N.~Gismondi, and M.~Novack,
\emph{Intermittent singular solutions of the stationary 2D Navier--Stokes equations in sharp Sobolev spaces},
arXiv:2506.00841, 2025.
\bibitem{Bourgain}
J.~Bourgain,
\emph{Fourier transform restriction phenomena for certain lattice subsets and applications to nonlinear evolution equations. I. Schr\"odinger equations},
Geom. Funct. Anal. \textbf{3} (1993), no.~2, 107--156.

\bibitem{BuckmasterVicol}
T.~Buckmaster and V.~Vicol,
\emph{Nonuniqueness of weak solutions to the Navier--Stokes equation},
Ann. of Math. (2) \textbf{189} (2019), no.~1, 101--144.

\bibitem{CheskidovHou}
A.~Cheskidov and H.~Hou,
\emph{On non-uniqueness of mild solutions and stationary singular solutions to the Navier--Stokes equations},
arXiv:2603.03666, 2026.

\bibitem{CheskidovLuo}
A.~Cheskidov and X.~Luo,
\emph{Sharp nonuniqueness for the Navier--Stokes equations},
Invent. Math. \textbf{229} (2022), no.~3, 987--1054.

\bibitem{Christ}
M.~Christ,
\emph{Nonuniqueness of weak solutions of the nonlinear Schr\"odinger equation},
arXiv:math/0503366, 2005.

\bibitem{CollianderOh}
J.~Colliander and T.~Oh,
\emph{Almost sure well-posedness of the cubic nonlinear Schr\"odinger equation below $L^2(\T)$},
Duke Math. J. \textbf{161} (2012), no.~3, 367--414.

\bibitem{DeLellisSzekelyhidi}
C.~De Lellis and L.~Sz\'ekelyhidi, Jr.,
\emph{The Euler equations as a differential inclusion},
Ann. of Math. (2) \textbf{170} (2009), no.~3, 1417--1436.

\bibitem{GMPR}
N.~Gismondi, K.~Ma, M.~Pathak, and A.~F.~Radu,
\emph{Non-unique solutions to the periodic gKdV equation},
arXiv:2606.06916, 2026.

\bibitem{Gromov}
M.~Gromov,
\emph{Partial differential relations},
Ergebnisse der Mathematik und ihrer Grenzgebiete (3), vol.~9,
Springer-Verlag, Berlin, 1986.

\bibitem{Gross}
E.~P.~Gross,
\emph{Structure of a quantized vortex in boson systems},
Nuovo Cimento (10) \textbf{20} (1961), 454--477.

\bibitem{GKO}
Z.~Guo, S.~Kwon, and T.~Oh,
\emph{Poincar\'e--Dulac normal form reduction for unconditional well-posedness of the periodic cubic NLS},
Comm. Math. Phys. \textbf{322} (2013), no.~1, 19--48.

\bibitem{GuoOh}
Z.~Guo and T.~Oh,
\emph{Non-existence of solutions for the periodic cubic NLS below $L^2$},
Int. Math. Res. Not. IMRN \textbf{2018} (2018), no.~6, 1656--1729.

\bibitem{HasegawaTappertI}
A.~Hasegawa and F.~Tappert,
\emph{Transmission of stationary nonlinear optical pulses in dispersive dielectric fibers. I. Anomalous dispersion},
Appl. Phys. Lett. \textbf{23} (1973), no.~3, 142--144.

\bibitem{HasegawaTappertII}
A.~Hasegawa and F.~Tappert,
\emph{Transmission of stationary nonlinear optical pulses in dispersive dielectric fibers. II. Normal dispersion},
Appl. Phys. Lett. \textbf{23} (1973), no.~4, 171--172.
\bibitem{Isett}
P.~Isett,
\emph{A proof of Onsager's conjecture},
Ann. of Math. (2) \textbf{188} (2018), no.~3, 871--963.
\bibitem{Katznelson}
Y.~Katznelson,
\emph{An introduction to harmonic analysis},
third ed., Cambridge Mathematical Library, Cambridge University Press, Cambridge, 2004.

\bibitem{Kuiper}
N.~H.~Kuiper,
\emph{On $C^1$-isometric imbeddings. I, II},
Nederl. Akad. Wetensch. Proc. Ser. A \textbf{58} = Indag. Math. \textbf{17} (1955),
545--556 and 683--689.

\bibitem{Molinet}
L.~Molinet,
\emph{On ill-posedness for the one-dimensional periodic cubic Schr\"odinger equation},
Math. Res. Lett. \textbf{16} (2009), no.~1, 111--120.

\bibitem{Nash}
J.~Nash,
\emph{$C^1$ isometric imbeddings},
Ann. of Math. (2) \textbf{60} (1954), no.~3, 383--396.

\bibitem{OhWangGWP}
T.~Oh and Y.~Wang,
\emph{Global well-posedness of the one-dimensional cubic nonlinear Schr\"odinger equation in almost critical spaces},
J. Differential Equations \textbf{269} (2020), no.~1, 612--640.

\bibitem{OhWangNF}
T.~Oh and Y.~Wang,
\emph{Normal form approach to the one-dimensional periodic cubic nonlinear Schr\"odinger equation in almost critical Fourier--Lebesgue spaces},
J. Anal. Math. \textbf{143} (2021), no.~2, 723--762.

\bibitem{Pitaevskii}
L.~P.~Pitaevskii,
\emph{Vortex lines in an imperfect Bose gas},
Sov. Phys. JETP \textbf{13} (1961), no.~2, 451--454.

\bibitem{Tsutsumi}
Y.~Tsutsumi,
\emph{$L^2$-solutions for nonlinear Schr\"odinger equations and nonlinear groups},
Funkcial. Ekvac. \textbf{30} (1987), 115--125.

\bibitem{ZakharovWaterWaves}
V.~E.~Zakharov,
\emph{Stability of periodic waves of finite amplitude on the surface of a deep fluid},
J. Appl. Mech. Tech. Phys. \textbf{9} (1968), no.~2, 190--194.

\bibitem{ZakharovShabatFocusing}
V.~E.~Zakharov and A.~B.~Shabat,
\emph{Exact theory of two-dimensional self-focusing and one-dimensional self-modulation of waves in nonlinear media},
Sov. Phys. JETP \textbf{34} (1972), no.~1, 62--69.

\bibitem{ZakharovShabatDefocusing}
V.~E.~Zakharov and A.~B.~Shabat,
\emph{Interaction between solitons in a stable medium},
Sov. Phys. JETP \textbf{37} (1973), no.~5, 823--828.

\bibitem{pathak2026nontrivial}
M.~Pathak,
\emph{Nontrivial weak solutions of the stationary KdV equation
in sharp \(L^p\) spaces},
arXiv preprint arXiv:2603.12555, 2026.

\bibitem{GismondiSlabs}
N.~Gismondi,
\emph{Nontrivial integrable weak stationary solutions to
active scalar equations with non-odd drift},
arXiv:2601.14592, 2026.


\end{thebibliography}
\end{document}